\documentclass[nohyperref]{article}

\usepackage{microtype}
\usepackage{graphicx}
\usepackage{subfigure}
\usepackage{booktabs} 

\usepackage{hyperref}

\usepackage[accepted]{icml2022}

\usepackage{amsmath}
\usepackage{amssymb}
\usepackage{mathtools}
\usepackage{amsthm}

\usepackage[capitalize,noabbrev]{cleveref}

\theoremstyle{plain}
\newtheorem{theorem}{Theorem}[section]
\newtheorem{proposition}[theorem]{Proposition}

\newtheorem{corollary}[theorem]{Corollary}
\theoremstyle{definition}
\newtheorem{definition}[theorem]{Definition}
\newtheorem{assumption}[theorem]{Assumption}
\theoremstyle{remark}

\usepackage[textsize=tiny]{todonotes}

\usepackage{thmtools}
\usepackage{thm-restate}
\usepackage{float}
\usepackage{listings}

\usepackage{color}
\usepackage{colortbl}
\definecolor{lightgray}{gray}{0.9}

\newcommand{\grad}{\operatorname{grad}}
\newcommand{\Hess}{\operatorname{Hess}}
\newcommand{\diam}{\operatorname{diam}}

\icmltitlerunning{Accelerated Gradient Methods for Geodesically Convex Optimization}

\begin{document}

\twocolumn[
\icmltitle{Accelerated Gradient Methods for Geodesically Convex Optimization: Tractable Algorithms and Convergence Analysis}



\icmlsetsymbol{equal}{*}

\begin{icmlauthorlist}
\icmlauthor{Jungbin Kim}{yyy}
\icmlauthor{Insoon Yang}{yyy}
\end{icmlauthorlist}

\icmlaffiliation{yyy}{Department of Electrical and Computer Engineering, Seoul National University, Seoul, South Korea}

\icmlcorrespondingauthor{Insoon Yang}{insoonyang@snu.ac.kr}

\icmlkeywords{Optimization}

\vskip 0.3in
]



\printAffiliationsAndNotice{}  

\begin{abstract}
	We propose computationally tractable accelerated first-order methods for Riemannian optimization, extending the Nesterov accelerated gradient (NAG) method. For both geodesically convex and geodesically strongly convex objective functions, our algorithms are shown to have the same iteration complexities as those for the NAG method on Euclidean spaces, under only standard assumptions. To the best of our knowledge, the proposed scheme is the first fully accelerated method for geodesically convex optimization problems.
	Our convergence analysis makes use of novel metric distortion lemmas as well as carefully designed potential functions. A connection with the continuous-time dynamics for modeling Riemannian acceleration in \citep{alimisis2020continuous} is also identified by letting the stepsize tend to zero. We validate our theoretical results through numerical experiments. 
\end{abstract}

\section{Introduction}
\label{sec:introduction}

\renewcommand{\arraystretch}{1.2}

\begin{table*}[ht]
	\caption{Iteration complexities (required number of iterations to obtain an $\epsilon$-approximate solution) for various accelerated methods on Riemannian manifolds. The notation $\tilde{O}(\cdot)$ and $O^{*}(\cdot)$ omits $\log (L/\epsilon)$ and $\log (L/\mu)$ factors, respectively \citep{martinezrubio2021global}. For our algorithms, the constant $\xi$ is defined as $\xi=\zeta+3(\zeta-\delta)$, where $\zeta$ and $\delta$ are defined in \cref{sec:assumption}. For the iteration complexity of RAGD \citep{zhang2018estimate}, $\frac{10}{9}$ is not regarded as a constant because this constant arises from their nonstandard assumption $d\left(x_{0},x^{*}\right)\leq\frac{1}{20\sqrt{\max\left\{ K_{\max},-K_{\min}\right\} }}\left(\frac{\mu}{L}\right)^{\frac{3}{4}}$.}
	\label{table:complexity}
	\vskip 0.15in
	\begin{center}
		\begin{tabular}{|c|c|c|c|}
			\hline 
			\rowcolor{lightgray}
			Algorithm & Objective function & Iteration complexity & Remark\\ \hline
			Algorithm~1 \citep{liu2017accelerated} & g-strongly convex & $O\left(\sqrt{L/\mu}\log\left(L/\epsilon\right)\right)$ & computationally intractable \\ \hline
			Algorithm~2 \citep{liu2017accelerated} & g-convex & $O\left(\sqrt{L/\epsilon}\right)$ & computationally  intractable \\ \hline
			RAGD \citep{zhang2018estimate} & g-strongly convex & $O\left((10/9)\sqrt{L/\mu}\log\left(L/\epsilon\right)\right)$ & nonstandard assumption \\ \hline
			Algorithm~1 \citep{ahn2020from} & g-strongly convex & $O^{*}\left(L/\mu + \sqrt{L/\mu}\log\left(\mu/\epsilon\right)\right)$ & eventually accelerated \\ \hline
			RAGDsDR \citep{alimisis2021momentum} & g-convex & $O\left(\sqrt{\zeta L/\epsilon}\right)$ & only in early stages \\ \hline
			\citep{martinezrubio2021global} & g-convex & $\tilde{O}\left(\sqrt{L/\epsilon}\right)$ & only for constant curvature \\ \hline
			\citep{martinezrubio2021global} & g-strongly convex & $O^{*}\left( \sqrt{L/\mu}\log\left(\mu/\epsilon\right)\right)$ & only for constant curvature \\ \hline
			RNAG-C (ours) & g-convex & $O\left(\xi\sqrt{L/\epsilon}\right)$ &  \\ \hline
			RNAG-SC (ours) & g-strongly convex & $O\left(\xi\sqrt{L/\mu}\log\left(L/\epsilon\right)\right)$ &  \\ \hline
		\end{tabular}
	\end{center}
	\vskip -0.1in
\end{table*}

\renewcommand{\arraystretch}{1}

We consider Riemannian optimization problems of the form
\begin{equation}
	\label{eq:problem}
	\min_{x \in N\subseteq M} \; f(x),
\end{equation}
where $M$ is a Riemannian manifold, $N$ is an open geodesically uniquely convex subset of $M$, and $f:N\rightarrow\mathbb{R}$ is a continuously differentiable \emph{geodesically convex} function.
Geodesically convex optimization is the Riemannian version of convex optimization  and has salient features such as every local minimum being a global minimum. 
More interestingly, some (constrained) nonconvex optimization problems defined in the Euclidean space can be considered geodesically convex optimization problems on appropriate Riemannian manifolds~\citep[Section~1]{vishnoi2018geodesic}. 
Geodesically convex optimization has a wide range of applications,  including covariance estimation~\citep{wiesel2012geodesic}, Gaussian mixture models~\citep{hosseini2015matrix,hosseini2020alternative}, matrix square root computation~\citep{sra2015matrix}, metric learning~\citep{zadeh2016geometric}, and optimistic likelihood calculation~\citep{nguyen2019calculating}. See \citep[Section~1.1]{zhang2016first} for more examples.

The iteration complexity theory for first-order algorithms is well known when $M=\mathbb{R}^n$. Given an initial point $x_0$, gradient descent (GD) updates the iterates as
\begin{equation}
	\label{eq:gd}
	\tag{GD}
	x_{k+1}=x_k-\gamma_k\grad f\left(x_k\right).
\end{equation}
For a convex and $L$-smooth objective function $f$, \ref{eq:gd} with $\gamma_k=\frac{1}{L}$ finds an $\epsilon$-approximate solution, i.e., $f\left(x_k\right)-f\left(x^{*}\right)\leq\epsilon$, in $O\left(\frac{L}{\epsilon}\right)$ iterations. For a $\mu$-strongly convex and $L$-smooth objective function $f$, \ref{eq:gd} with $\gamma_k=\frac{1}{L}$ finds an $\epsilon$-approximate solution in $O\left(\frac{L}{\mu}\log\frac{L}{\epsilon}\right)$ iterations.
A major breakthrough in first-order algorithms is the Nesterov accelerated gradient (NAG) method that achieves a faster convergence rate than \ref{eq:gd}~\cite{nesterov1983}.

Given an initial point $x_0=z_0$, the NAG scheme updates the iterates as
\begin{equation}
	\label{eq:nag}
	\tag{NAG}
	\begin{aligned}
		y_{k} & =x_{k}+\tau_k\left(z_{k}-x_{k}\right)\\
		x_{k+1} & =y_{k}-\alpha_k\grad f\left(y_{k}\right)\\
		z_{k+1} & =y_{k}+\beta_k\left(z_{k}-y_{k}\right)-\gamma_k\grad f\left(y_{k}\right).
	\end{aligned}
\end{equation}
For a convex and $L$-smooth function $f$, \ref{eq:nag} with $\tau_k=\frac{2}{k+2}$, $\alpha_k=\frac{1}{L}$, $\beta_k=1$, $\gamma_k=\frac{k+2}{2L}$ (NAG-C) finds an $\epsilon$-approximate solution in $O\left(\sqrt{\frac{L}{\epsilon}}\right)$ iterations \citep{tseng2008accelerated}. For a $\mu$-strongly convex and $L$-smooth objective function~$f$, \ref{eq:nag} with $\tau_k=\frac{\sqrt{\mu/L}}{1+\sqrt{\mu/L}}$, $\alpha_k=\frac{1}{L}$, $\beta_k=1-\sqrt{\frac{\mu}{L}}$, $\gamma_k=\sqrt{\frac{\mu}{L}}\frac{1}{\mu}$ (NAG-SC) finds an $\epsilon$-approximate solution in $O\left(\sqrt{\frac{L}{\mu}}\log\frac{L}{\epsilon}\right)$ iterations \citep{nesterov2018lectures}.

Considering the problem~\eqref{eq:problem} for any Riemannian manifold $M$, \citep{zhang2016first} successfully generalizes the complexity analysis of \ref{eq:gd} to Riemannian gradient descent (RGD),
\begin{equation}
	\label{eq:rgd}
	\tag{RGD}
	x_{k+1}=\exp_{x_k}\left(-\gamma_k\grad f\left(x_k\right)\right),
\end{equation}
using a lower bound $K_{\min}$ of the sectional curvature and an upper bound $D$ of $\diam(N)$.
For completeness, we provide a potential-function analysis in \cref{app:rgd} to show that \ref{eq:rgd} with a fixed stepsize has the same iteration complexity as \ref{eq:gd}.

However, it is still unclear whether a reasonable generalization of NAG to the Riemannian setting is possible with strong theoretical guarantees. When studying the global complexity of Riemannian optimization algorithms, it is common to assume that the sectional curvature of $M$ is bounded below by $K_{\min}$ and bounded above by $K_{\max}$  to prevent the manifold from being overly curved. Unfortunately, \citep{criscitiello2021negative,hamilton2021nogo} show that even when sectional curvature is bounded, achieving global acceleration is impossible in general. Thus, one might need another common assumption, an upper bound $D$ of $\diam(N)$. 
This motivates our central question:

\begin{center}
	\emph{Can we design computationally tractable accelerated first-order methods on Riemannian manifolds when the sectional curvature and the diameter of the domain are bounded?}
\end{center}

In the literature, there are  some partial answers but no full answer to this question (see \cref{table:complexity} and \cref{sec:related}). In this paper, we provide a complete answer via new first-order algorithms, which we call the Riemannian Nesterov accelerated gradient (RNAG) method. 
We show that acceleration is possible on Riemannian manifolds for both geodesically convex (g-convex) and geodesically strongly convex (g-strongly convex) cases whenever the bounds $K_{\min}$, $K_{\max}$, and $D$ are available. The main contributions of this work can be summarized as follows:

\begin{itemize}
	\item Generalizing Nesterov's scheme, we propose \ref{eq:rnag}, a first-order method for Riemannian optimization. We provide two specific algorithms: RNAG-C (\cref{alg:rnag-c}) for minimizing g-convex functions and RNAG-SC (\cref{alg:rnag-sc}) for minimizing g-strongly convex functions. 
	Both algorithms call one gradient oracle per iteration. Our algorithms are computationally tractable in the sense that they only involve exponential maps, logarithm maps, parallel transport, and operations in tangent spaces.
	In particular, RNAG-C can be interpreted as a variant of NAG-C with high friction in \citep[Section~4.1]{su2014} (see \cref{app:su}). 
	
	\item Given the bounds $K_{\min}$, $K_{\max}$, and $D$, we prove that RNAG-C has an $O\left(\sqrt{\frac{L}{\epsilon}}\right)$ iteration complexity (\cref{cor:rnag-c}), and that RNAG-SC has an $O\left(\sqrt{\frac{L}{\mu}}\log\frac{L}{\epsilon}\right)$ iteration complexity (\cref{cor:rnag-sc}). The crucial steps of the proofs are constructing potential functions as \eqref{eq:potential} and handling metric distortion using \cref{lem:distortion1} and \cref{lem:distortion2}. To the best of our knowledge, this is the first proof for full acceleration in the g-convex case.
	
	\item We identify a connection between our algorithms and the ODEs for modeling Riemannian acceleration in \citep{alimisis2020continuous} by letting the stepsize tend to zero. 
	This analysis confirms the accelerated convergence of our algorithms through the lens of continuous-time flows. 
	
\end{itemize}

\section{Related Work}
\label{sec:related}

Given a bound $D$ for $\diam(N)$, \citep{liu2017accelerated} proposed accelerated methods for both g-convex and g-strongly convex cases. Their algorithms have the same iteration complexities as \ref{eq:nag} but require a solution to a nonlinear equation at every iteration, which could be as difficult as solving the original problem in general. Given $K_{\min}$, $K_{\max}$, and $d\left(x_0,x^{*}\right)$, \citep{zhang2018estimate} proposed a computationally tractable algorithm for the g-strongly convex case and showed that their algorithm achieves the iteration complexity $O\left(\frac{10}{9}\sqrt{\frac{L}{\mu}}\log\frac{L}{\epsilon}\right)$ when $d\left(x_{0},x^{*}\right)\leq\frac{1}{20\sqrt{\max\left\{ K_{\max},-K_{\min}\right\} }}\left(\frac{\mu}{L}\right)^{\frac{3}{4}}$. Given only $K_{\min}$ and $K_{\max}$, \citep{ahn2020from} considered the g-strongly convex case. Although full acceleration is not guaranteed, the authors proved that their algorithm eventually achieves acceleration in later stages. 
Given $K_{\min}$, $K_{\max}$, and $D$, \citep{alimisis2021momentum} proposed a momentum method for the g-convex case. 
They showed that their algorithm achieves acceleration in early stages. Although this result is not as strong as full acceleration, their theoretical guarantee is meaningful in practical situations. \citep{martinezrubio2021global} focused on manifolds with  constant sectional curvatures, namely a subset of the hyperbolic space or sphere. Their algorithm is accelerated, but it is not straightforward to generalize their argument to any manifolds. 
Beyond the g-convex setting, \citep{criscitiello2020accelerated} studied accelerated methods for nonconvex problems. \citep{lezcanocasado2020adaptive} studied adaptive and momentum-based methods using the trivialization framework in \citep{lezcanocasado2019trivializations}. Further works on accelerated Riemannian optimization can be found in \citep[Section~1.6]{criscitiello2021negative}.

Another line of research takes the perspective of continuous-time dynamics as in the Euclidean counterpart~\citep{su2014,wibisono2016,wilson2021lyapunov}. 
For both g-convex and g-strongly convex cases, \citep{alimisis2020continuous} proposed ODEs that can model accelerated methods on Riemannian manifolds given $K_{\min}$ and $D$. \citep{duruisseaux2021variational} extended this result and developed a variational framework. 
Time-discretization methods for such ODEs on Riemannian manifolds have recently been of considerable interest as well~\citep{duruisseaux2021accelerated,francca2021optimization,duruisseaux2022accelerated}. 

While many positive results have been obtained for accelerated Riemannian optimization, there are also a few negative results \citep{hamilton2021nogo} and \citep{criscitiello2021negative}, showing that achieving full acceleration for Riemannian optimization is impossible in general. Because their results involve a growing diameter of domain and most of the positive results assume that the diameter of domain is bounded by a constant $D$, the negative result is not contradictory but complementary to the positive results. This indicates that the assumption of bounding the domain by a constant is necessary for achieving  full acceleration. See \cref{sec:discussion} for a detailed discussion.

\section{Preliminaries}

\subsection{Background}

A Riemannian manifold $(M,g)$ is a real smooth manifold equipped with a Riemannian metric $g$ which assigns to each $p\in M$ a positive-definite inner product $g_{p}(v,w)=\langle v,w\rangle_p=\langle v,w\rangle$ on the tangent space $T_p M$. The inner product $g_p$ induces the norm $\left\Vert v\right\Vert _{p}=\Vert v\Vert$ defined as $\sqrt{\langle v,v\rangle_{p}}$ on $T_pM$. The tangent bundle $TM$ of $M$ is defined  as $TM=\sqcup_{p\in M}T_{p}M$. For $p,q\in M$, the geodesic distance $d(p,q)$ between $p$ and $q$ is the infimum of the length of all piecewise continuously differentiable curves from $p$ to $q$. For nonempty set $N\subseteq M$, the diameter $\diam(N)$ of $N$ is defined as $\diam(N)=\sup_{p,q\in N}d(p,q)$.

For a smooth function $f:M\rightarrow\mathbb{R}$, the Riemannian gradient $\grad f(x)$ of $f$ at $x$ is defined as the tangent vector in $T_xM$ satisfying
\[
\langle \grad f(x),v\rangle=df(x)[v],
\]
where $df(x):T_x M\rightarrow \mathbb{R}$ is the differential of $f$ at $x$. 
Let $I := [0,1]$.
A geodesic $\gamma: I \to M$ is a smooth curve of locally minimum length with zero acceleration.\footnote{The mathematical definition of acceleration is provided in \cref{app:background}.} 
In particular, straight lines in $\mathbb{R}^n$ are geodesics. The exponential map at $p$ is defined as, for $v\in T_pM$,
\[
\exp_p(v)=\gamma_v(1),
\]
where $\gamma_v:I\rightarrow M$ is the geodesic satisfying $\gamma_v(0)=p$ and $\gamma_v'(0)=v$. In general, $\exp_p$ is only defined on a neighborhood of $0$ in $T_pM$. It is known that $\exp_p$ is a diffeomorphism in some neighborhood $U$ of $0$. Thus, its inverse  is well defined and is called the logarithm map $\log_x:\exp_p(U)\rightarrow T_pM$.  For a smooth curve $\gamma:I\rightarrow M$ and $t_0,t_1\in I$, the parallel transport $\Gamma(\gamma)_{t_0}^{t_1}:T_{\gamma(t_0)}M\rightarrow T_{\gamma(t_1)}M$ is a way of transporting vectors from $T_{\gamma(t_0)}M$ to $T_{\gamma(t_1)}M$ along $\gamma$.\footnote{{The definition using covariant derivatives is contained in \cref{app:background}.}}   
When $\gamma$ is a geodesic, we 
let $\Gamma_p^q:T_pM\rightarrow T_qM$
denote the parallel transport from $T_pM$ to $T_qM$.

A subset $N$ of $M$ is said to be geodesically uniquely convex if for every $x,y\in N$, there exists a unique geodesic $\gamma:[0,1]\rightarrow M$ such that $\gamma(0)=x$, $\gamma(1)=y$, and $\gamma(t)\in N$ for all $t\in[0,1]$. Let $N$ be a geodesically uniquely convex subset of $M$. A function $f:N\rightarrow\mathbb{R}$ is said to be geodesically convex if $f\circ\gamma:[0,1]\rightarrow\mathbb{R}$ is convex for each geodesic $\gamma:[0,1]\rightarrow M$ whose image is in $N$. When $f$ is geodesically convex, we have
\[
f(y)\geq f(x)+\left\langle\grad f(x),\log_x(y)\right\rangle.
\]
Let $N$ be an open geodesically uniquely convex subset of $M$, and $f:N\rightarrow\mathbb{R}$ be a continuously differentiable function. We say that $f$ is geodesically $\mu$-strongly convex for $\mu>0$ if
\[
f(y)\geq f(x)+\left\langle\grad f(x),\log_x(y)\right\rangle+\frac{\mu}{2}\left\Vert \log_x(y)\right\Vert^2
\]
for all $x,y\in N$. We say that $f$ is geodesically $L$-smooth if 
\[
f(y)\leq f(x)+\left\langle\grad f(x),\log_x(y)\right\rangle+\frac{L}{2}\left\Vert \log_x(y)\right\Vert^2
\]
for all $x,y\in N$.

For additional notions from Riemannian geometry that are used in our analysis, we refer the reader to \cref{app:background} as well as the textbooks~\citep{lee2018riemannian,petersen,boumal2020intromanifolds}. 

\subsection{Assumptions}
\label{sec:assumption}

In this subsection, we present the assumptions that are imposed throughout the paper. 
\begin{assumption}
	\label{ass:1}
	The domain $N$ is an open geodesically uniquely convex subset of $M$. The diameter of the domain is bounded as $\mathrm{diam}(N) \leq D < \infty$. The sectional curvature inside $N$ is bounded below by $K_{\min}$ and bounded above by $K_{\max}$. If $K_{\max} > 0$, we further assume that $D < \frac{\pi}{\sqrt{K_{\max}}}$.
\end{assumption}

\cref{ass:1} implies that the exponential map $\exp_x$ is a diffeomorphism for any $x\in N$ \citep{alimisis2021momentum}.

\begin{assumption}
	\label{ass:2}
	The objective function $f: N \to \mathbb{R}$ is continuously differentiable and geodesically $L$-smooth. Moreover, $f$ is bounded below, and has minimizers, all of which lie in $N$.  A global minimizer is denoted by $x^{*}$.
\end{assumption}

\begin{assumption}
	\label{ass:3}
	All the iterates $x_k$ and $y_k$ are well-defined on the manifold $M$ remain in $N$.
\end{assumption}

Although Assumption~\ref{ass:3} is common in the literature \citep{zhang2018estimate,ahn2020from,alimisis2021momentum}, it is desirable to relax or remove it. 
We leave the extension as a future research topic.

To implement our algorithms, we also assume that we can compute (or approximate) exponential maps, logarithmic maps, and parallel transport. For many manifolds in practical applications, these maps are implemented in libraries such as~\citep{townsend2016pymanopt}.

We define the constants $\zeta\geq1$ and $\delta\leq1$ as
\begin{align*}
	\zeta & =\begin{cases}
		\sqrt{-K_{\min}}D\coth\left(\sqrt{-K_{\min}}D\right), & \text{if }K_{\min}<0\\
		1, & \text{if }K_{\min}\geq0
	\end{cases}\\
	\delta & =\begin{cases}
		1, & \text{if }K_{\max}\leq0\\
		\sqrt{K_{\max}}D\cot\left(\sqrt{K_{\max}}D\right), & \text{if }K_{\max}>0.
	\end{cases}
\end{align*}
These constants naturally arise from the Rauch comparison theorem \citep[Theorem~11.7]{lee2018riemannian} \citep[Theorem~6.4.3]{petersen}, and many known methods on Riemannian manifolds have a convergence rate depending on some of these constants \cite{alimisis2020continuous,alimisis2021momentum,zhang2016first}. Note that we can set $\zeta=\delta=1$ when $M=\mathbb{R}^n$.

\section{Algorithms}

In this section, we first generalize Nesterov's scheme to the Riemannian setting and then design specific algorithms for both g-convex and g-strongly convex cases. In \cite{ahn2020from,zhang2018estimate} \ref{eq:nag} is generalized to a three-step algorithm on a Riemannian manifold as
\begin{equation}
	\label{eq:ahn_scheme}
	\begin{aligned}
		y_{k} & =\exp_{x_{k}}\left(\tau_k\log_{x_{k}}\left(z_{k}\right)\right)\\
		x_{k+1} & =\exp_{y_{k}}\left(-\alpha_k\grad f\left(y_{k}\right)\right)\\
		z_{k+1} & =\exp_{y_{k}}\left(\beta_k\log_{y_{k}}\left(z_{k}\right)-\gamma_k\grad f\left(y_{k}\right)\right).
	\end{aligned}
\end{equation}
However, it is more natural to define the iterates $z_k$ in the tangent bundle $TM$, instead of in $M$.\footnote{The scheme \eqref{eq:ahn_scheme} always uses $z_k$ after mapping it to $TM$ via logarithm maps. The proof of convergence  needs the value of $f$ at $x_k$ and $y_k$. Thus, these iterates (but not $z_k$) should be defined in $M$. Considering continuous-time interpretation, the role of $z_k$ is similar to the role of velocity vector $\dot{X}$, which is defined in $TM$ (see \cref{sec:continuous}).} Thus, we propose another scheme that involves iterates in $TM$ without using $z_k$. To associate tangent vectors in different tangent spaces, we use parallel transport, which is a way to transport vectors from one tangent space to another.

\begin{figure}[tp]
	
	\begin{center}
		\centerline{\includegraphics[width=\linewidth]{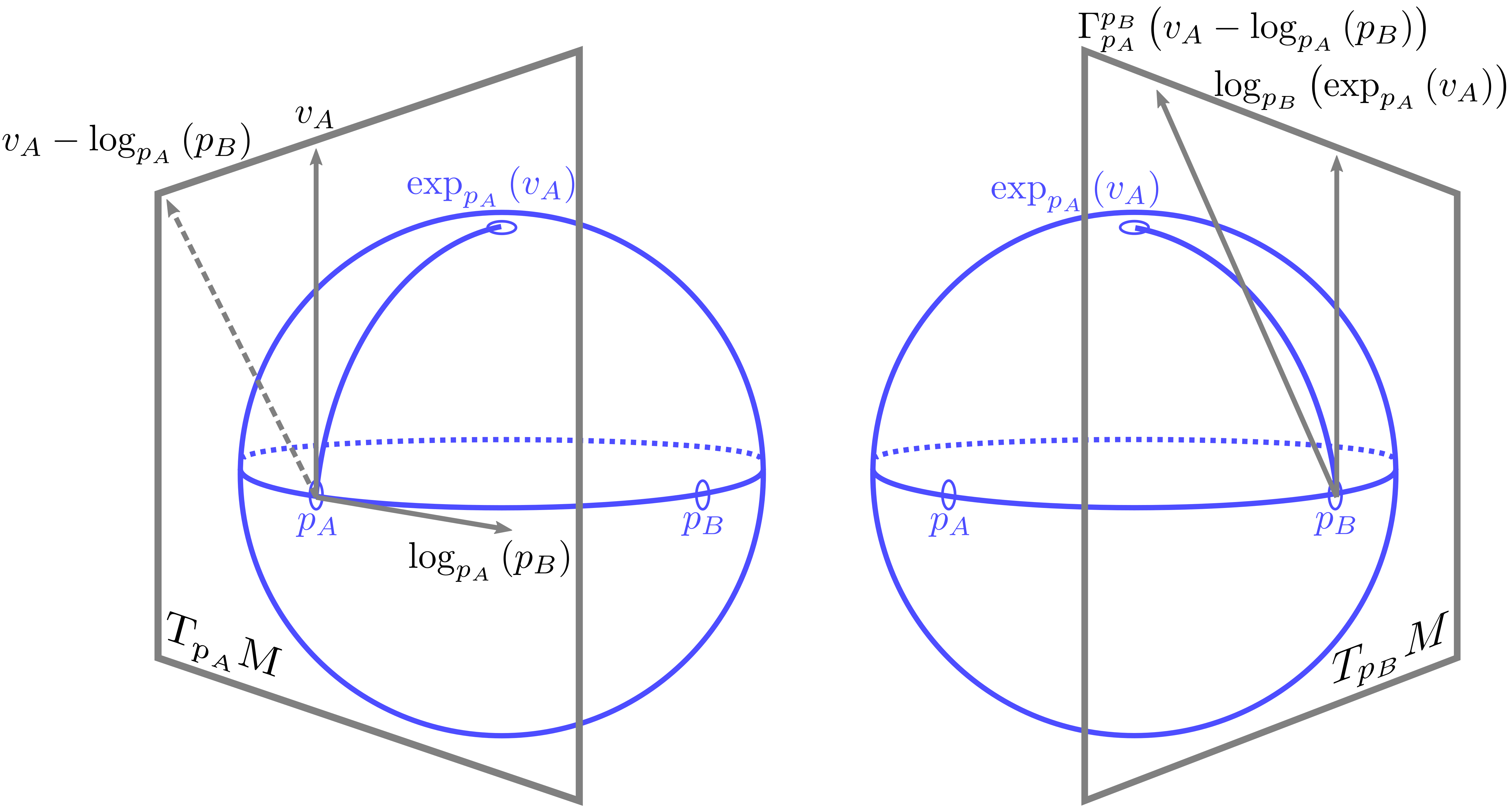}}
		\caption{Illustration of the maps $v_A\mapsto\Gamma_{p_{A}}^{p_{B}}\left(v_A-\log_{p_{A}}\left(p_{B}\right)\right)$ and $v_A\mapsto\log_{p_B}\left(\exp_{p_A}\left(v_A\right)\right)$.}
		\label{fig:picture}
	\end{center}
	\vskip -0.2in
\end{figure}

Given $z_k\in M$ in \eqref{eq:ahn_scheme}, we define the iterates $v_{k}=\log_{y_{k}}\left(z_{k}\right)$, $\bar{v}_{k}=\log_{x_{k}}\left(z_{k}\right)$, and $\bar{\bar{v}}_{k}=\log_{y_{k-1}}\left(z_{k}\right)$ in the tangent bundle $TM$. It is straightforward to check that the following scheme is equivalent to \eqref{eq:ahn_scheme}:
\begin{equation}
	\label{eq:ahn_scheme2}
	\begin{aligned}
		y_{k} & =\exp_{x_{k}}\left(\tau_k\bar{v}_{k}\right)\\
		x_{k+1} & =\exp_{y_{k}}\left(-\alpha_k\grad f\left(y_{k}\right)\right)\\
		v_{k} & =\log_{y_{k}}\left(\exp_{x_{k}}\left(\bar{v}_{k}\right)\right)\\
		\bar{\bar{v}}_{k+1} & =\beta_kv_{k}-\gamma_k\grad f\left(y_{k}\right)\\
		\bar{v}_{k+1} & =\log_{x_{k+1}}\left(\exp_{y_{k}}\left(\bar{\bar{v}}_{k+1}\right)\right).
	\end{aligned}
\end{equation}
In \eqref{eq:ahn_scheme2}, the third and last steps associate tangent vectors in different tangent spaces using the map $T_{p_A}M\rightarrow T_{p_B}M;\ v_A\mapsto\log_{p_B}\left(\exp_{p_A}\left(v_A\right)\right)$. We change these steps by using the map $v_A\mapsto\Gamma_{p_{A}}^{p_{B}}\left(v_A-\log_{p_{A}}\left(p_{B}\right)\right)$ instead. Technically, this modification allows us to use \cref{lem:distortion2} when handling metric distortion in our convergence analysis. With the change, we obtain the following scheme, which we call RNAG:
\begin{equation}
	\label{eq:rnag}
	\tag{RNAG}
	\begin{aligned}
		y_{k} & =\exp_{x_{k}}\left(\tau_k\bar{v}_{k}\right)\\
		x_{k+1} & =\exp_{y_{k}}\left(-\alpha_k\grad f\left(y_{k}\right)\right)\\
		v_{k} & =\Gamma_{x_{k}}^{y_{k}}\left(\bar{v}_{k}-\log_{x_{k}}\left(y_{k}\right)\right)\\
		\bar{\bar{v}}_{k+1} & =\beta_kv_{k}-\gamma_k\grad f\left(y_{k}\right)\\
		\bar{v}_{k+1} & =\Gamma_{y_{k}}^{x_{k+1}}\left(\bar{\bar{v}}_{k+1}-\log_{y_{k}}\left(x_{k+1}\right)\right).
	\end{aligned}
\end{equation}
Because \ref{eq:rnag} only involves exponential maps, logarithm maps, parallel transport, and operations in tangent spaces, this scheme is computationally tractable, unlike the scheme in \citep{liu2017accelerated}, which involves a nonlinear operator.
Note that \ref{eq:rnag} is different from the scheme \eqref{eq:ahn_scheme} because the maps $v_A\mapsto\log_{p_B}\left(\exp_{p_A}\left(v_A\right)\right)$ and $v_A\mapsto\Gamma_{p_{A}}^{p_{B}}\left(v_A-\log_{p_{A}}\left(p_{B}\right)\right)$ are not equivalent in general (see Figure~\ref{fig:picture}).

\begin{algorithm}[tp]
	\caption{RNAG-C}
	\label{alg:rnag-c}
	\begin{algorithmic}
		\STATE {\bfseries Input:} initial point $x_0$, parameters $\xi$ and $T>0$, step size $s\leq\frac{1}{L}$
		\STATE Initialize $\bar{v}_{0}=0\in T_{x_{0}}M$.
		\STATE Set $\lambda_k=\frac{k+2\xi+T}{2}$.
		\FOR{$k=0$ {\bfseries to} $K-1$}
		\STATE $y_{k}=\exp_{x_{k}}\left(\frac{\xi}{\lambda_{k}+(\xi-1)}\bar{v}_{k}\right)$
		\STATE $x_{k+1}=\exp_{y_{k}}\left(-s\grad f\left(y_{k}\right)\right)$
		\STATE $v_{k}=\Gamma_{x_{k}}^{y_{k}}\left(\bar{v}_{k}-\log_{x_{k}}\left(y_{k}\right)\right)$
		\STATE $\bar{\bar{v}}_{k+1}=v_{k}-\frac{s\lambda_k}{\xi}\grad f\left(y_{k}\right)$
		\STATE $\bar{v}_{k+1}=\Gamma_{y_{k}}^{x_{k+1}}\left(\bar{\bar{v}}_{k+1}-\log_{y_{k}}\left(x_{k+1}\right)\right)$
		\ENDFOR
		\STATE {\bfseries Output:} $x_K$
	\end{algorithmic}
\end{algorithm}

\begin{algorithm}[tp]
	\caption{RNAG-SC}
	\label{alg:rnag-sc}
	\begin{algorithmic}
		\STATE {\bfseries Input:} initial point $x_0$, parameter $\xi$, step size $s\leq\frac{1}{L}$
		\STATE Initialize $\bar{v}_{0}=0\in T_{x_{0}}M$.
		\STATE Set $q=\mu s$.
		\FOR{$k=0$ {\bfseries to} $K-1$}
		\STATE $y_{k}=\exp_{x_{k}}\left(\frac{\sqrt{\xi q}}{1+\sqrt{\xi q}}\bar{v}_{k}\right)$
		\STATE $x_{k+1}=\exp_{y_{k}}\left(-s\grad f\left(y_{k}\right)\right)$
		\STATE $v_{k}=\Gamma_{x_{k}}^{y_{k}}\left(\bar{v}_{k}-\log_{x_{k}}\left(y_{k}\right)\right)$
		\STATE $\bar{\bar{v}}_{k+1}=\left(1-\sqrt{\frac{q}{\xi}}\right)v_{k}+\sqrt{\frac{q}{\xi}}\left(-\frac{1}{\mu}\grad f\left(y_{k}\right)\right)$
		\STATE $\bar{v}_{k+1}=\Gamma_{y_{k}}^{x_{k+1}}\left(\bar{\bar{v}}_{k+1}-\log_{y_{k}}\left(x_{k+1}\right)\right)$
		\ENDFOR
		\STATE {\bfseries Output:} $x_K$
	\end{algorithmic}
\end{algorithm}

\begin{figure}[tp]
	
	\begin{center}
		\centerline{\includegraphics[width=0.9\columnwidth]{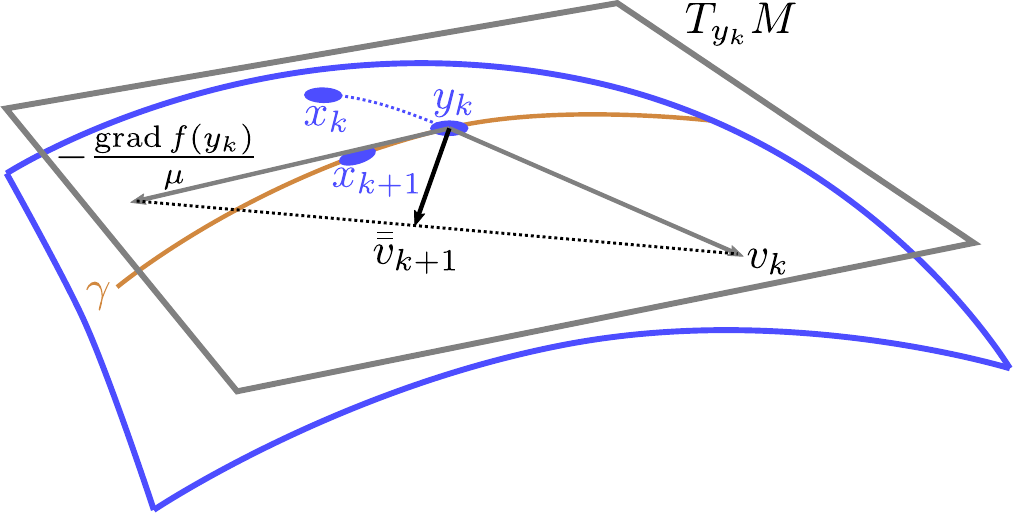}}
		\caption{Illustration of RNAG-SC.}
		\label{fig:algorithm}
	\end{center}
	\vskip -0.2in
\end{figure}

By carefully choosing the parameters $\tau_k$, $\alpha_k$, $\beta_k$ and $\gamma_k$, we finally obtain two algorithms, RNAG-C (Algorithm~\ref{alg:rnag-c}) for the g-convex case, and RNAG-SC (Algorithm~\ref{alg:rnag-sc}) for the g-strongly convex case. In particular, we can interpret RNAG-C as a slight variation of NAG-C with high friction \citep[Section~4.1]{su2014} with the friction parameter $r=1+2\xi$. See \cref{app:su} for a detailed interpretation. Note that we recover NAG-C and NAG-SC 
from these algorithms
when $M=\mathbb{R}^n$ and $\xi=1$.
Figure~\ref{fig:algorithm} is an illustration of some steps of RNAG-SC, where the curve $\gamma$ is a geodesic with $\gamma(0)=y_k$ and $\gamma'(0)=\grad f\left(y_k\right)$.

\section{Convergence Analysis}

\subsection{Metric distortion lemma}

To handle a potential function involving squared norms in tangent spaces, we need to compare distances in different tangent spaces.

\begin{proposition}
	\label{prop:alimisis}
	\citep[Lemma 2]{alimisis2020continuous}
	Let $\gamma$ be a smooth curve whose image is in $N$. Then, we have
	\[
	\delta\left\Vert \gamma'(t)\right\Vert ^{2}\leq\left\langle D_{t}\log_{\gamma(t)}(x),-\gamma'(t)\right\rangle \leq\zeta\left\Vert \gamma'(t)\right\Vert ^{2}.
	\]
\end{proposition}

In the proposition above, $D_t$ is a covariant derivative along the curve (see \cref{app:background}). Using this proposition, we obtain the following lemma.

\begin{restatable}{lemma}{distortiona}
	\label{lem:distortion1}
	Let $p_{A},p_{B},x\in N$ and $v_A\in T_{p_A}M$. If there is $r\in[0,1]$ such that $\log_{p_A}\left(p_B\right)=rv_A$, then we have
	\begin{align*}
		& \left\Vert v_{B}-\log_{p_{B}}\left(x\right)\right\Vert_{p_B} ^{2}+(\zeta-1)\left\Vert v_{B}\right\Vert_{p_B} ^{2}\\
		& \leq\left\Vert v_{A}-\log_{p_{A}}\left(x\right)\right\Vert_{p_A} ^{2}+(\zeta-1)\left\Vert v_{A}\right\Vert_{p_A} ^{2},
	\end{align*}
	where $v_B=\Gamma_{p_{A}}^{p_{B}}\left(v_{A}-\log_{p_{A}}\left(p_{B}\right)\right)\in T_{p_B}M$
\end{restatable}

In particular, when $r=1$, \cref{lem:distortion1} recovers a weaker version of \citep[Lemma~5]{zhang2016first}. We can further generalize this lemma as follows:

\begin{restatable}{lemma}{distortionb}
	\label{lem:distortion2}
	Let $p_{A},p_{B},x\in N$ and $v_{A}\in T_{p_{A}}M$. Define $v_{B}=\Gamma_{p_{A}}^{p_{B}}\left(v_{A}-\log_{p_{A}}\left(p_{B}\right)\right)\in T_{p_{B}}M$.
	If there are $a,b\in T_{p_{A}}M$, and $r\in(0,1)$ such that $v_{A}=a+b$
	and $\log_{p_{A}}\left(p_{B}\right)=rb$, then we have
	\begin{align*}
		& \left\Vert v_{B}-\log_{p_{B}}\left(x\right)\right\Vert_{p_B} ^{2}+(\xi-1)\left\Vert v_{B}\right\Vert_{p_B} ^{2}\\
		& \leq\left\Vert v_{A}-\log_{p_{A}}\left(x\right)\right\Vert_{p_A} ^{2}+(\xi-1)\left\Vert v_{A}\right\Vert_{p_A} ^{2}\\
		& \quad+\frac{\xi-\delta}{2}\left(\frac{1}{1-r}-1\right)\left\Vert a\right\Vert_{p_A} ^{2}
	\end{align*}
	for $\xi\geq\zeta$.
\end{restatable}

As $\exp_{p_A}\left(v_A\right)\neq\exp_{p_B}\left(v_B\right)$ (see \cref{fig:picture}), our lemma does not compare the projected distance\footnote{For $u,v,w\in M$, the projected distance between $v$ and $w$ with respect to $u$ is defined as $\left\Vert \log_u(v)-\log_u(w)\right\Vert ^{\color{red}\raisebox{1ex}{\underline{\smash{\raisebox{-1ex}{\scriptsize2}}}}}$ \citep[Definition~3.1]{ahn2020from}.} between points on the manifold, unlike \citep[Theorem~10]{zhang2018estimate} and \citep[Lemma~4.1]{ahn2020from}. The proofs of \cref{lem:distortion1} and \cref{lem:distortion2} can be found in Appendix~\ref{app:distortion}.

\subsection{Main results}

We now prove the iteration complexities of RNAG-C and RNAG-SC using potential functions of the form
\begin{equation}
	\label{eq:potential}
	\begin{aligned}
		\phi_{k} & =A_{k}\left(f\left(x_{k}\right)-f\left(x^{*}\right)\right)\\
		& \quad+B_{k}\left(\left\Vert \bar{v}_{k}-\log_{x_{k}}\left(x^{*}\right)\right\Vert _{x_{k}}^{2}+(\xi-1)\left\Vert \bar{v}_{k}\right\Vert _{x_{k}}^{2}\right).
	\end{aligned}
\end{equation}
The term $(\xi-1)\left\Vert \bar{v}_{k}\right\Vert_{x_k} ^{2}$  is novel compared with the potential function in \citep{ahn2020from}, and it measures the kinetic energy~\citep{wibisono2016}. 
Intuitively, this potential makes sense because a large $\xi$ means high friction (see \cref{app:su} and \cref{sec:continuous}). This term is useful when handling metric distortion.

\subsubsection{The geodesically convex case}

For the g-convex case, we use a potential function defined as
\begin{equation}
	\label{eq:potential_rnag-c}
	\begin{aligned}
		\phi_{k} & =s\lambda_{k-1}^2\left(f\left(x_{k}\right)-f\left(x^{*}\right)\right)\\
		& \quad+\frac{\xi}{2}\left\Vert \bar{v}_{k}-\log_{x_{k}}\left(x^{*}\right)\right\Vert ^{2}+\frac{\xi(\xi-1)}{2}\left\Vert \bar{v}_{k}\right\Vert ^{2}.
	\end{aligned}
\end{equation}
The following theorem shows that this potential function is decreasing when the parameters $\xi$ and $T$ are chosen appropriately.

\begin{restatable}{theorem}{rnagcthm}
	\label{thm:rnag-c}
	Let $f$ be a g-convex and geodesically $L$-smooth function. If the parameters $\xi$ and $T$ of RNAG-C satisfy $\xi\geq\zeta$ and
	\begin{align*}
		& \frac{\xi-\delta}{2}\left(\frac{1}{1-\xi/\lambda_{k}}-1\right)\\
		& \leq(\xi-\zeta)\left(\frac{1}{\left(1-\xi/\left(\lambda_{k}+\xi-1\right)\right)^{2}}-1\right)
	\end{align*}
	for all $k\geq 0$, then the iterates of RNAG-C satisfy $\phi_{k+1}\leq \phi_k$ for all $k\geq0$, where $\phi_k$ is defined as \eqref{eq:potential_rnag-c}.
\end{restatable}

In particular, we can show that the parameters $\xi=\zeta+3(\zeta-\delta)$ and $T=4\xi$ satisfy the condition in \cref{thm:rnag-c}. In this case, the monotonicity of the potential function yields
\[
f\left(x_{k}\right)-f\left(x^{*}\right)\leq\frac{1}{s\lambda_{k-1}^{2}}\phi_{k}\leq\frac{1}{s\lambda_{k-1}^{2}}\phi_{0}.
\]
Thus, RNAG-C achieves acceleration. The result is summarized in the following corollary.

\begin{restatable}{corollary}{rnagccor}
	\label{cor:rnag-c}
	Let $f$ be a g-convex and geodesically $L$-smooth function. Then, RNAG-C with parameters $\xi=\zeta+3(\zeta-\delta)$, $T=4\xi$ and step size $s=\frac{1}{L}$ finds an $\epsilon$-approximate solution in $O\left(\xi\sqrt{\frac{L}{\epsilon}}\right)$ iterations.
\end{restatable}

This result implies that  the iteration complexity of RNAG-C is the same as that of NAG-C because $\xi$ is a constant. 
The proofs of \cref{thm:rnag-c} and \cref{cor:rnag-c} are contained in Appendix~\ref{app:rnag-c}.

\subsubsection{The geodesically strongly convex case}

For the g-strongly convex case, we consider a potential function defined as
\begin{equation}
	\label{eq:potential_rnag-sc}
	\begin{aligned}
		\phi_{k}&=\left(1-\sqrt{\frac{q}{\xi}}\right)^{-k}\Biggl(f\left(x_{k}\right)-f\left(x^{*}\right)\\&\quad+\frac{\mu}{2}\left\Vert {\color{red}\bar{v}_{k}}-\log_{\color{red}x_{k}}\left(x^{*}\right)\right\Vert ^{2}+\frac{\mu(\xi-1)}{2}\left\Vert {\color{red}\bar{v}_{k}}\right\Vert ^{2}\Biggr).
	\end{aligned}
\end{equation}
This potential function is also shown to be decreasing under appropriate conditions on $\xi$ and $s$.

\begin{restatable}{theorem}{rnagscthm}
	\label{thm:rnag-sc}
	Let $f$ be a geodesically $\mu$-strongly convex and geodesically $L$-smooth function. If the step size $s$ and the parameter $\xi$ of RNAG-SC satisfy $\xi\geq\zeta$, $\sqrt{\xi q}<1$, and
	\begin{align*}
		& \frac{\xi-\delta}{2}\left(\frac{1}{1-\sqrt{\xi q}}-1\right)\left(1-\sqrt{\frac{q}{\xi}}\right)^{2}-\sqrt{\xi q}\left(1-\sqrt{\frac{q}{\xi}}\right)\\
		& \leq(\xi-\zeta){\color{red}\left(1-\sqrt{\frac{q}{\xi}}\right)}\left(\frac{1}{\left(1-\sqrt{\xi q}/\left(1+\sqrt{\xi q}\right)\right)^{2}}-1\right),
	\end{align*}
	then the iterates of RNAG-SC satisfy $\phi_{k+1}\leq \phi_k$ for all $k\geq0$, where $\phi_k$ is defined as \eqref{eq:potential_rnag-sc}.
\end{restatable}

In particular, the parameters $\xi=\zeta+3(\zeta-\delta)$ and $s=\frac{1}{9\xi L}$ satisfy the condition in \cref{thm:rnag-sc}. In this case, by monotonicity of the potential function, we have
\[ 
f\left(x_{k}\right)-f\left(x^{*}\right)\leq\left(1-\sqrt{\frac{q}{\xi}}\right)^{k}\phi_{k}\leq\left(1-\sqrt{\frac{q}{\xi}}\right)^{k}\phi_{0},
\]
which implies that RNAG-SC achieves acceleration. The following corollary summarizes the result. 

\begin{restatable}{corollary}{rnagsccor}
	\label{cor:rnag-sc}
	Let $f$ be a geodesically $\mu$-strongly convex and geodesically $L$-smooth function. Then, RNAG-SC with parameter $\xi=\zeta+3(\zeta-\delta)$ and step size $s=\frac{1}{9\xi L}$ finds an $\epsilon$-approximate solution in $O\left(\xi\sqrt{\frac{L}{\mu}}\log\left(\frac{L}{\epsilon}\right)\right)$ iterations.
\end{restatable}

Because $\xi$ is a constant, the iteration complexity of RNAG-SC is the same as that of NAG-SC. 
The proofs of \cref{thm:rnag-sc} and \cref{cor:rnag-sc} can be found in Appendix~\ref{app:rnag-sc}.

\section{Continuous-Time Interpretation}
\label{sec:continuous}

In this section, we identify a connection to the ODEs for modeling Riemannian acceleration in~\citep[Equations~2 and 4]{alimisis2020continuous}.
Specifically, 
following the informal arguments in \citep[Section~2]{su2014} and \citep[Section~4.8]{daspremont2021acceleration}, we obtain ODEs by taking the limit $s\rightarrow 0$ in our schemes. 
The detailed analysis is contained in \cref{app:continuous}. 

For a sufficiently small $s$, the Euclidean geometry is valid as only a sufficiently small subset of $M$ is considered.
Thus, we informally assume $M=\mathbb{R}^n$ for simplicity. 
We can show that the iterations of RNAG-C satisfy
\begin{align*}
	& \frac{y_{k+1}-y_{k}}{\sqrt{s}}\\
	& =\frac{\lambda_{k}-1}{\lambda_{k+1}+(\xi-1)}\frac{y_{k}-y_{k-1}}{\sqrt{s}}\\
	& \quad-\frac{\lambda_{k+1}}{\lambda_{k+1}+(\xi-1)}\sqrt{s}\grad f\left(y_{k}\right)\\
	& \quad+\frac{\lambda_{k}-1}{\lambda_{k+1}+(\xi-1)}\sqrt{s}\left(\grad f\left(y_{k-1}\right)-\grad f\left(y_{k}\right)\right).
\end{align*}
We introduce a smooth curve $y(t)$ that is approximated by the iterates of RNAG-C as $y(t)\approx y_{t/\sqrt{s}}=y_k$ with $k=\frac{t}{\sqrt{s}}$. Using the Taylor expansion, we have
\begin{align*}
	\frac{y_{k+1}-y_{k}}{\sqrt{s}} & =\dot{y}(t)+\frac{\sqrt{s}}{2}\ddot{y}(t)+o\left(\sqrt{s}\right),\\
	\frac{y_{k}-y_{k-1}}{\sqrt{s}} & =\dot{y}(t)-\frac{\sqrt{s}}{2}\ddot{y}(t)+o\left(\sqrt{s}\right),\\
	\sqrt{s}\grad f\left(y_{k-1}\right) & =\sqrt{s}\grad f\left(y_{k}\right)+o\left(\sqrt{s}\right).
\end{align*}
Letting $s\rightarrow0$ yields the ODE\footnote{When $M\neq\mathbb{R}^n$, we replace $\ddot{y}$ with the acceleration $\nabla_{\dot{y}}\dot{y}=D_t\dot{y}$, where $D_t$ is a covariant derivative along the curve $y$ (see \cref{app:background}).}
\begin{equation}
	\label{eq:rnag-ode-c}
	\nabla_{\dot{y}}\dot{y}+\frac{1+2\xi}{t}\dot{y}+\grad f(y)=0,
\end{equation}
where the covariant derivative $\nabla_{\dot{y}}\dot{y}=D_t\dot{y}$ is a natural extension of the second derivative $\ddot{y}$ (see \cref{app:background}). 

In the g-strongly convex case, we can show that the iterations of RNAG-SC satisfy
\begin{align*}
	& \frac{y_{k+1}-y_{k}}{\sqrt{s}}\\
	& =\frac{1-\sqrt{q/\xi}}{1+\sqrt{\xi q}}\frac{y_{k}-y_{k-1}}{\sqrt{s}}-\frac{1+\sqrt{ q/\xi}}{1+\sqrt{\xi q}}\sqrt{s}\grad f\left(y_{k}\right)\\
	& \quad+\frac{1-\sqrt{ q/\xi}}{1+\sqrt{\xi q}}\sqrt{s}\left(\grad f\left(y_{k-1}\right)-\grad f\left(y_{k}\right)\right).
\end{align*}
Through a similar limiting process, we obtain the following ODE:
\begin{equation}
	\label{eq:rnag-ode-sc}
	\nabla_{\dot{y}}\dot{y}+\left(\frac{1}{\sqrt{\xi}}+\sqrt{\xi}\right)\sqrt{\mu}\dot{y}+\grad f(y)=0.
\end{equation}

Replacing the parameter $\xi$ in the coefficients of our ODEs with $\zeta$,  we recover \citep[Equations~2 and 4]{alimisis2020continuous}.
Because $\xi\geq\zeta$, the continuous-time acceleration results \citep[Theorems~5 and 7]{alimisis2020continuous}
are valid for our ODEs as well. 
Thus, this analysis confirms the accelerated convergence of our algorithms through the lens of continuous-time flows.

In both ODEs,   the parameter $\xi \geq \zeta$ appears in the coefficient of the friction term $\dot{X}$, increasing with $\xi$. Intuitively, this makes sense because $\zeta$ is large for an ill-conditioned domain, where $-K_{\min}$ and $D$ are large and thus metric distortion is more severe  (where one might want to decrease the effect of momentum).

\section{Experiments}
\label{sec:experiments}

\begin{figure*}[ht]
	\centering
	\subfigure[Rayleigh quotient maximization]{\includegraphics[width=0.33\textwidth]{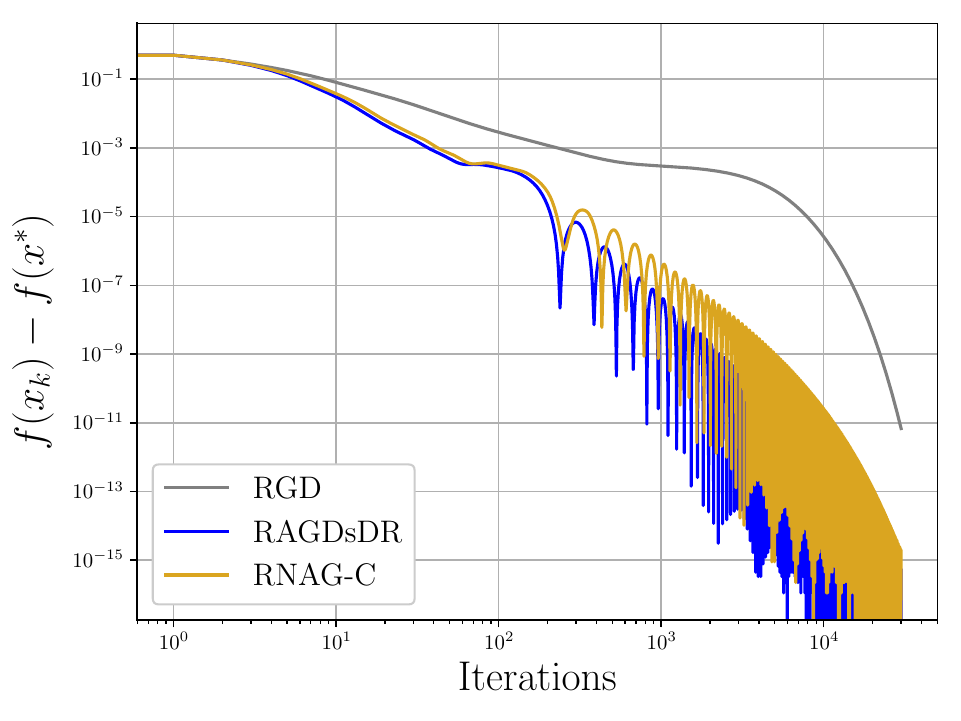}\label{fig:rayleigh}}
	\subfigure[Karcher mean of SPD matrices]{\includegraphics[width=0.33\textwidth]{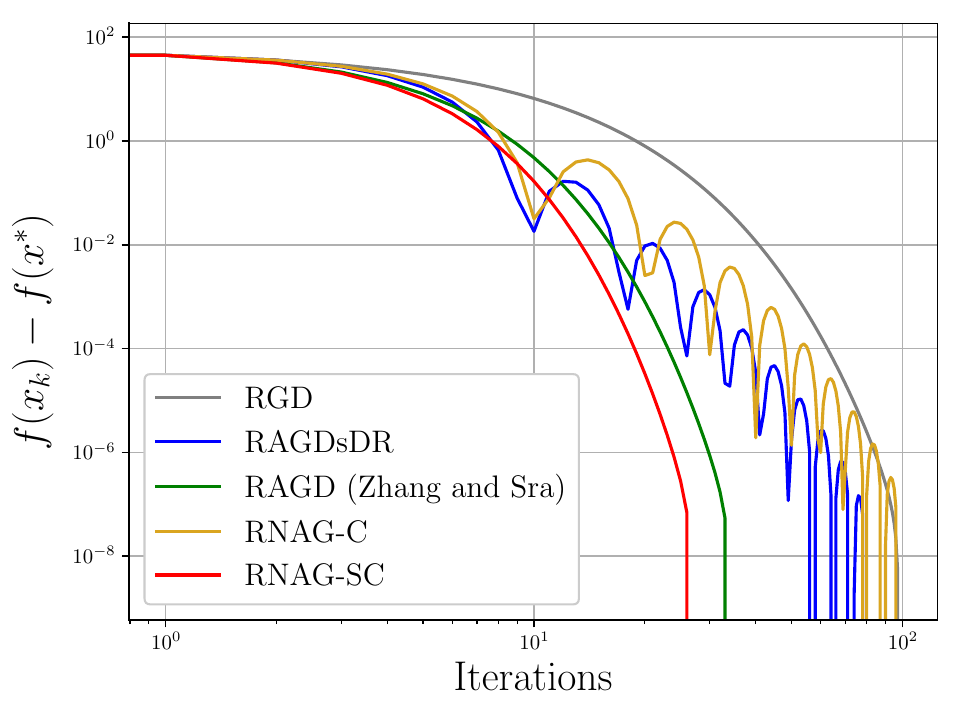}\label{fig:spd}}
	\subfigure[Karcher mean on hyperbolic space]{\includegraphics[width=0.33\textwidth]{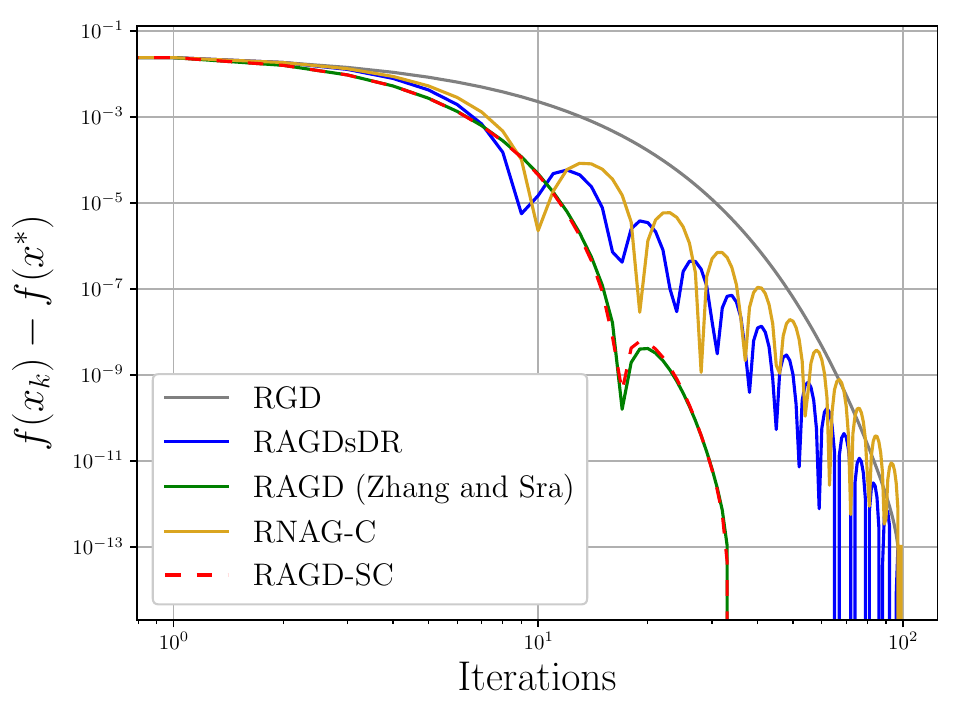}\label{fig:hyper}}
	\caption{Performances of various Riemannian optimization algorithms on the Rayleigh quotient maximization problem and the Karcher mean problem.}\label{fig:experiments}
\end{figure*}

In this section, we examine the performance of our algorithms on the Rayleigh quotient maximization problem and the Karcher mean problem. To implement the geometry of manifolds, we used the Python libraries Pymanopt \citep{townsend2016pymanopt} and Geomstats \citep{geomstats}. For comparison, we use the known accelerated algorithms RAGD \citep{zhang2018estimate} for the g-strongly convex case and RNAGsDR with no line search \citep{alimisis2021momentum} for the g-convex case.  The source code of our RNAG implementation is available online.\footnote{\url{https://github.com/jungbinkim1/RNAG}}

We set the input parameters as $\zeta=1$ for implementing RAGDsDR, and $\xi=1$ for implementing our algorithms. The stepsize was chosen as $s=\frac{1}{L}$ in our algorithms.

{\bf Rayleigh quotient maximization.}

Given a real $d\times d$ symmetric matrix $A$, we consider the problem
\[
\min_{x\in\mathbb{S}^{d-1}}f(x)=-\frac{1}{2}x^{\top}Ax.
\]
on the unit $(d-1)$-sphere on $\mathbb{S}^{d-1}$. For this manifold, we set $K_{\min}=K_{\max}=1$. We let $d=1000$ and $A=\frac{1}{2}\left(B+B^{\top}\right)$, where the entries of $B\in\mathbb{R}^{d\times d}$ were randomly generated by the Gaussian distribution $N(0,1/d)$. We have the smoothness parameter $L=\lambda_{\max}-\lambda_{\min}$ by the following proposition. 

\begin{restatable}{proposition}{experimenta}
	\label{prop:rayleigh}
	The function $f$ is geodesically $\left(\lambda_{\max}-\lambda_{\min}\right)$-smooth, where $\lambda_{\max}$ and $\lambda_{\min}$ are the largest and smallest eigenvalues of $A$, respectively.
\end{restatable}

The proof can be found in \cref{app:exp}. The result is shown in \cref{fig:rayleigh}.  We observe that RNAG-C outperforms RGD and is comparable to RAGDsDR, a known accelerated method for the g-convex case.

{\bf Karcher mean of SPD matrices.}
When $K_{\max}\leq0$, the Karcher mean \citep{karcher1977riemannian} of the points $p_i\in M$ for $i=1,\ldots,n$, is defined as the solution of 
\begin{equation}
	\label{eq:karcher}
	\min_{x\in M}f(x)=\frac{1}{2n}\sum_{i=1}^{n}d\left(x,p_{i}\right)^{2}.
\end{equation}

{The following proposition shows that one can set the strong convexity parameter as $\mu=1$.}

\begin{restatable}{proposition}{experimentb}
	
	The function $f$ is geodesically $1$-strongly convex.
\end{restatable}

The proof can be found in \cref{app:exp}. We consider this problem on the manifold $\mathcal{P}(d)\subseteq\mathbb{R}^{d\times d}$ of symmetric positive definite matrices endowed with the Riemannian metric $\langle X,Y\rangle_{P}=\operatorname{Tr}\left(P^{-1}XP^{-1}Y\right)$. It is known that one can set $K_{\min}=-\frac{1}{2}$ and $K_{\max}=0$ \citep[Appendix~I]{criscitiello2020accelerated}. We set the dimension and the number of matrices as $d=100$ and $n=50$. The matrices $p_i$ were randomly generated using Matrix Mean Toolbox \citep{bini2013} with condition number ${10}^6$. We set the smoothness parameter as $L=10$. The result is shown in \cref{fig:spd}. We observe that RNAG-SC and RAGD \citep{zhang2018estimate} perform significantly better than \ref{eq:rgd}. The performances of RNAG-C and RAGDsDR are only slightly better than that of \ref{eq:rgd} in early stages. This result makes sense because $f$ is g-strongly convex and well-conditioned.

{\bf Karcher mean on hyperbolic space.}
We consider the problem \eqref{eq:karcher} on the hyperbolic space $\mathbb{H}^d$ with the hyperboloid model $\mathbb{H}^{d}=\left\{ x\in\mathbb{R}^{d+1}:-x_{d+1}^{2}+\sum_{k=1}^{d}x_{k}^{2}=-1\right\} $. For this manifold, we can set $K_{\min}=K_{\max}=-1$. We set the dimension and the number of points as $d=1000$ and $n=10$. First $d$ entries of each point $p_i$ are randomly generated by the Gaussian distribution $N(0,1/d)$. We set the smoothness parameter as $L=10$. The result is similar to that of the previous example, and is shown in \cref{fig:hyper}.

\section{Discussion}\label{sec:discussion}

In this paper, we have proposed novel computationally tractable first-order methods that achieve Riemannian acceleration for both g-convex and g-strongly convex objective functions whenever the constants $K_{\min}$, $K_{\max}$, and $D$ are available. The iteration complexities of RNAG-C and RNAG-SC match those of their Euclidean counterparts. The continuous-time analysis of our algorithms provides an intuitive interpretation of the parameter $\xi$ as a measurement of friction, which is higher when the domain manifold is more ill-conditioned. In fact, the iteration complexities of our algorithms depend on the parameter $\xi\geq\zeta$, which is affected by the values of the constants $K_{\min}$, $K_{\max}$, and $D$. When $\zeta$ is large (i.e., $-K_{\min}$ and $D$ are large), we have a worse guarantee. A possible future direction is to study the effect of the constants $K_{\min}$, $K_{\max}$, and $D$ on the complexities of Riemannian optimization algorithms tightly.

{\bf Comparison with \citep{liu2017accelerated}.}
The algorithms in \citep{liu2017accelerated} achieve acceleration with only standard assumption. However, to implement the operator $\mathbb{S}:(y_{k-1},x_k,x_{k-1})\mapsto y_k$ in \citep[Algorithm~1]{liu2017accelerated}, one needs to solve the following nonlinear equation 	at each iteration:
\begin{align*}
	& (1-\sqrt{\mu/L})\Gamma_{y_{k}}^{y_{k-1}}\log_{y_{k}}\left(x_{k}\right)-\beta\Gamma_{y_{k}}^{y_{k-1}}\grad f\left(y_{k}\right)\\
	& = (1-\sqrt{\mu/L})^{3/2}\log_{y_{k-1}}\left(x_{k-1}\right).
\end{align*}
It is unclear whether this equation is solvable in a tractable way or even feasible as noted in \citep{ahn2020from}. On the other hand, our algorithms involve only operations in tangent spaces and the exponential map, logarithm map, and parallel transport. Thus, our algorithms are computationally tractable for various manifolds in practice, where the operations above are implementable.

{\bf Comparison with \cite{criscitiello2021negative}.}
It is natural to ask how our positive result is not contradictory to the negative result in \cite{criscitiello2021negative}. To clarify this, we provide the following two reasons:

(i) We assume that the diameter $\diam (N)$ of the domain $N$ is bounded, which is a more restrictive condition than their assumption that the distance $d\left(x_0,x^{*}\right)$ is bounded.

(ii) We assume that the diameter $\diam (N)$ is bounded by a fixed constant $D$. Thus, in \cref{cor:rnag-c} and \cref{cor:rnag-sc}, $\xi$ does not depend on other parameters such as $\mu$ and $L$. In contrast, \citep[Theorem~1.3]{criscitiello2021negative} introduces a bound $\frac34 r$ of $d\left(x_0,x^{*}\right)$ by letting $r$ be the solution of $\kappa = 12r\sqrt{-K_{\min}}+9$), thus $r\sqrt{-K_{\min}}$ grows with $\kappa=L/\mu$. A similar discussion can be found in \citep[Remark~29]{martinezrubio2021global}.

We believe that the second one is the main reason for our positive results coexist with their negative results. As mentioned in \cref{sec:related}, their result is not contradictory but complementary to our results.

\section*{Acknowledgements}

We thank the anonymous reviewers for their insightful suggestions	and Dr. Antonio Orvieto and the co-authors of \citep{alimisis2021momentum}  for allowing us to use a part of their  code.
We are also grateful to Dr. Ken'ichiro Tanaka for pointing out an error in Eq. (6) in the previous version. 	
This work was supported in part by Samsung Electronics, the National Research Foundation of Korea funded by MSIT(2020R1C1C1009766), and the Information and Communications Technology Planning and Evaluation (IITP) grant funded by MSIT(2022-0-00124, 2022-0-00480).

\bibliography{reference}
\bibliographystyle{icml2022}

\newpage
\appendix
\onecolumn
\section{Background}
\label{app:background}

\begin{definition}
	A smooth vector field $V$ is a smooth map from $M$ to $TM$ such that $p\circ V$ is the identity map, where $p:TM\rightarrow M$ is the projection. The collection of all smooth vector fields on $M$ is denoted by $\mathfrak{X}(M)$.
\end{definition}

\begin{definition}
	Let $\gamma:I\rightarrow M$ be a smooth curve. A smooth vector field $V$ along $\gamma$ is a smooth map from $I$ to $TM$ such that $V(t)\in T_{\gamma(t)}M$ for all $t\in I$. The collection of all smooth vector fields along $\gamma$ is denoted by $\mathfrak{X}(\gamma)$.
\end{definition}

\begin{proposition}[Fundamental theorem of Riemannian geometry]
	There exists a unique operator
	\[
	\nabla:\mathfrak{X}(M)\times\mathfrak{X}(M)\rightarrow\mathfrak{X}(M):\ (U,V)\mapsto\nabla_UV
	\]
	satisfying the following properties for any $U,V,W\in \mathfrak{X}(M)$, smooth functions $f,g$ on $M$, and $a,b\in\mathbb{R}$:
	\begin{enumerate}
		\item $\nabla_{fU+gW}V=f\nabla_UV+g\nabla_WV$
		\item $\nabla_U(aV+bW)=a\nabla_UV+b\nabla_UW$
		\item $\nabla_U(fV)=(Uf)V+f\nabla_UV$
		\item $[U,V]=\nabla_UV-\nabla_VU$
		\item $U\langle V,W\rangle=\langle\nabla_UV,W\rangle+\langle V,\nabla_UW\rangle$,
	\end{enumerate}
	where $[\cdot,\cdot]$ denotes the Lie bracket. The operator $\nabla$  is called the Levi-Civita connection or the Riemannian connection. The field $\nabla_UV$ is called the covariant derivative of $V$ along $U$.
\end{proposition}

From now on, we always assume that $M$ is equipped with the Riemannian connection $\nabla$.

\begin{proposition}
	\citep[Section~8.11]{boumal2020intromanifolds}
	For any smooth vector fields $U,V$ on $M$, the vector field $\nabla_UV$ at $x$ depends on $U$ only through $U(x)$. Thus, we can write $\nabla_uV$ to mean $(\nabla_UV)(x)$ for any $U\in\mathfrak{X}(M)$ such that $U(x)=u$, without ambiguity.
\end{proposition}

For a smooth function $f:M\rightarrow\mathbb{R}$, $\grad f$ is a smooth vector field.

\begin{definition}
	\citep[Section~8.11]{boumal2020intromanifolds}
	The Riemannian Hessian of a smooth function $f$ on $M$ at $x\in M$ is a self-adjoint linear operator $\Hess f(x):T_xM\rightarrow T_xM$ defined as
	\[
	\Hess f(x)[u]=\nabla_u\grad f.
	\]
\end{definition}

\begin{proposition}
	\citep[Section~8.12]{boumal2020intromanifolds}
	Let $c:I\rightarrow M$ be a smooth curve. There exists a unique operator $D_t:\mathfrak{X}(c)\rightarrow\mathfrak{X}(c)$ satisfying the following properties for all $Y,Z\in\mathfrak{X}(c)$, $U\in\mathfrak{X}(M)$, a smooth function $g$ on $I$, and $a,b\in\mathbb{R}$:
	\begin{enumerate}
		\item $D_t(aY+bZ)=aD_tY+bD_tZ$
		\item $D_t(gZ)=g'Z+gD_tZ$
		\item $(D_t(U\circ c))(t)=\nabla_{c'(t)}U$ for all $t\in I$
		\item $\frac{d}{dt}\langle Y,Z\rangle=\langle D_tY,Z\rangle+\langle Y,D_tZ\rangle$.
	\end{enumerate}
	This operator is called the (induced) covariant derivative along the curve $c$.
\end{proposition}

We define the acceleration of a smooth curve $\gamma$ as the vector field $D_t\gamma'$ along $\gamma$. Now, we can define the parallel transport using  covariant derivatives.

\begin{definition}
	\citep[Section~10.3]{boumal2020intromanifolds}
	A vector field $Z\in\mathfrak{X}(c)$ is parallel if $D_tZ=0$.
\end{definition}

\begin{proposition}
	\citep[Section~10.3]{boumal2020intromanifolds}
	For any smooth curve $c:I\rightarrow M$, $t_0\in I$ and $u\in T_{c(t_0)}M$, there exists a unique parallel vector field $Z\in\mathfrak{X}(c)$ such that $Z(t_0)=u$.
\end{proposition}

\begin{definition}
	\citep[Section~10.3]{boumal2020intromanifolds}
	Given a smooth curve $c$ on $M$, the parallel transport of tangent vectors at $c(t_0)$ to the tangent space at $c(t_1)$ along $c$,
	\[
	\Gamma(c)_{t_0}^{t_1}:T_{c(t_0)}M\rightarrow T_{c(t_1)}M,
	\]
	is defined by $\Gamma(c)_{t_0}^{t_1}(u)=Z(t_1)$, where $Z\in\mathfrak{X}(c)$ is the unique parallel vector field such that $Z(t_0)=u$.
\end{definition}

\begin{proposition}
	\citep[Section~10.3]{boumal2020intromanifolds}
	The parallel transport operator $\Gamma(c)_{t_0}^{t_1}$ is linear. Also, $\Gamma(c)_{t_1}^{t_2}\circ\Gamma(c)_{t_0}^{t_1}=\Gamma(c)_{t_0}^{t_2}$ and $\Gamma(c)_{t}^{t}$ is the identity. In particular, the inverse of $\Gamma(c)_{t_0}^{t_1}$ is $\Gamma(c)_{t_1}^{t_0}$. The parallel transport is an isometry, that is,
	\[
	\langle u,v\rangle_{c\left(t_0\right)}=\left\langle \Gamma(c)_{t_0}^{t_1}(u),\Gamma(c)_{t_0}^{t_1}(v)\right\rangle_{c\left(t_1\right)}.
	\]
\end{proposition}

\begin{proposition}
	\label{prop:parallel}
	\citep[Section~10.3]{boumal2020intromanifolds}
	Consider a smooth curve $c:I\rightarrow M$. Given a vector field $Z\in\mathfrak{X}(c)$, we have
	\[
	D_tZ(t)=\lim_{h\rightarrow 0}\frac{\Gamma(c)_{t+h}^tZ(t+h)-Z(t)}{h}.
	\]
\end{proposition}

\section{Comparison between RNAG-C and High-Friction NAG-C}
\label{app:su}

In this section, we review high-friction NAG-C in \citep[Section~4.1]{su2014}, and compare it to RNAG-C. For $r\geq3$, they designed the gerenalized NAG-C with high friction as
\begin{align*}
	x_{k} & =y_{k-1}-s\grad f\left(y_{k-1}\right)\\
	y_{k} & =x_{k}+\frac{k-1}{k+r-1}\left(x_{k}-x_{k-1}\right).
\end{align*}
Introducing the third sequence as $z_{k}=y_{k}+\frac{k}{r-1}\left(y_{k}-x_{k}\right)$, we can rewrite this method as
\begin{equation}
	\tag{NAG-C-HF}
	\label{eq:nag-c-hf}
	\begin{aligned}
		y_{k} & =x_{k}+\frac{r-1}{k+r-1}\left(z_{k}-x_{k}\right)\\
		x_{k+1} & =y_{k}-s\grad f\left(y_{k}\right)\\
		z_{k+1} & =z_{k}-\frac{k+r-1}{r-1}s\grad f\left(y_{k}\right).
	\end{aligned}
\end{equation}
Note that we can recover NAG-C by letting $r=3$. The iterates of \ref{eq:nag-c-hf} satisfy
\[
f\left(x_{k}\right)-f\left(x^{*}\right)\leq\frac{(r-1)^{2}\left\Vert x_{0}-x^{*}\right\Vert ^{2}}{2s(k+r-2)^{2}}\leq\frac{(r-1)^{2}\left\Vert x_{0}-x^{*}\right\Vert ^{2}}{2s(k-2)^{2}}
\]
for $s\leq\frac{1}{L}$ \citep[Theorem~6]{su2014}. Thus, we have $f\left(x_k\right)-f\left(x^{*}\right)\leq\epsilon$ whenever
\[
(k-2)^{2}\geq\frac{(r-1)^{2}\left\Vert x_{0}-x^{*}\right\Vert ^{2}}{2s\epsilon}.
\]
In particular, when $s=\frac{1}{L}$ and $r=1+2\xi$, we have the Iteration complexity $O\left(\xi\sqrt{\frac{L}{\epsilon}}\right)$.

For comparison, we write RNAG-C in Euclidean space as
\begin{equation}
	\label{eq:rnag-c-euc}
	\begin{aligned}
		y_{k} & =x_{k}+\frac{2\xi}{k+2\xi+(T+2\xi-2)}\left(z_{k}-x_{k}\right)\\
		x_{k+1} & =y_{k}-s\grad f\left(y_{k}\right)\\
		z_{k+1} & =z_{k}-\frac{k+2\xi+T}{2\xi}s\grad f\left(y_{k}\right).
	\end{aligned}
\end{equation}
One can see that the algorithm \eqref{eq:rnag-c-euc} is similar to that of \ref{eq:nag-c-hf} with $r=1+2\xi$, where the only difference occurs in constants that can be ignored as $k$ grows. Note that both algorithms have the same iteration complexity $O\left(\xi\sqrt{\frac{L}{\epsilon}}\right)$ even when we do not ignore the effect of $\xi$, and lead to the same ODE \citep[Section~4.1]{su2014}
\[
\ddot{y}+\frac{1+2\xi}{t}\dot{y}+\grad f(y)=0.
\]

\section{Proofs of \cref{lem:distortion1} and \cref{lem:distortion2}}
\label{app:distortion}

\begin{proposition}
	\label{prop:alimisis2}
	\citep[Lemma~12]{alimisis2020continuous}
	Let $\gamma$ be a smooth curve whose image is in $N$, then
	\[
	\frac{d}{dt}\left\Vert \log_{\gamma(t)}(x)\right\Vert ^{2}=2\left\langle D_{t}\log_{\gamma(t)}(x),\log_{\gamma(t)}(x)\right\rangle =2\left\langle \log_{\gamma(t)}(x),-\gamma'(t)\right\rangle .
	\]
\end{proposition}

\distortiona*

\begin{proof}
	By geodesic unique convexity of $N$, there is a unique geodesic $\gamma$ such that $\gamma(0)=p_{A}$ and $\gamma(r)=p_{B}$ whose image lies in $N$. We can check that $\gamma'(0)=v_A$.\footnote{Consider the geodesic $c:t\mapsto\gamma(rt)$. Then $c(0)=p_A$ and $c(1)=p_B$. By definition of the exponential map, $c'(0)=\log_{p_A}\left(p_B\right)=rv_A$. Combining this equality with $c'(0)=r\gamma'(0)$ gives the desired result.} Define the vector field $V(t)$ along $\gamma$ as $V(t)=\Gamma(\gamma)_{0}^{t}\left(v_{A}-t\gamma'(0)\right)$. Then, we can check that $V(t)=(1-t)\gamma'(t)$ and $V'(t)=-\gamma'(t)$.\footnote{A similar argument as in the previous footnote shows the first equality. The second equality follows from \cref{prop:parallel} and the fact that $\gamma'(t)$ is parallel along $\gamma$.} Define the function $w:[0,r]\rightarrow\mathbb{R}$ as $w(t)=\left\Vert \log_{\gamma(t)}(x)-V(t)\right\Vert ^{2}$, It follows from \cref{prop:alimisis} and \cref{prop:alimisis2} that
	\begin{align*}
		\frac{d}{dt}w(t) & =2\left\langle D_{t}\left(\log_{\gamma(t)}(x)-V(t)\right),\log_{\gamma(t)}(x)-V(t)\right\rangle \\
		& =2\left\langle D_{t}\log_{\gamma(t)}(x),\log_{\gamma(t)}(x)\right\rangle -2\left\langle D_{t}\log_{\gamma(t)}(x),V(t)\right\rangle -2\left\langle D_{t}V(t),\log_{\gamma(t)}(x)\right\rangle +2\left\langle D_{t}V(t),V(t)\right\rangle \\
		& =2\left\langle D_{t}\log_{\gamma(t)}(x),\log_{\gamma(t)}(x)\right\rangle -2(1-t)\left\langle D_{t}\log_{\gamma(t)}(x),\gamma'(t)\right\rangle +2\left\langle \gamma'(t),\log_{\gamma(t)}(x)\right\rangle +2\left\langle D_{t}V(t),V(t)\right\rangle \\
		& =2(1-t)\left\langle D_{t}\log_{\gamma(t)}(x),-\gamma'(t)\right\rangle +2\left\langle D_{t}V(t),V(t)\right\rangle \\
		& \leq2(1-t)\zeta\left\Vert \gamma'(t)\right\Vert ^{2}+2\left\langle D_{t}V(t),V(t)\right\rangle \\
		& =-2\zeta\left\langle -\gamma'(t),(1-t)\gamma'(t)\right\rangle +2\left\langle D_{t}V(t),V(t)\right\rangle \\
		& =-2(\zeta-1)\left\langle D_{t}V(t),V(t)\right\rangle \\
		& =-(\zeta-1)\left(\frac{d}{dt}\left\Vert V(t)\right\Vert ^2\right).
	\end{align*}
	Integrating both sides from $0$ to $r$ gives
	\[
	w(r)-w(0)\leq\int_{0}^{r}-(\zeta-1)\left(\frac{d}{dt}\left\Vert V(t)\right\Vert ^2\right)\,dt =-(\zeta-1)\left(\left\Vert V(r)\right\Vert ^{2}-\left\Vert V(0)\right\Vert ^{2}\right).
	\]
	This completes the proof.
\end{proof}

\distortionb*

\begin{proof}
	Define $\gamma$, $V$, $w$ as in the proof of \cref{lem:distortion1}. As in the proof of \cref{lem:distortion1}, we can check that $\gamma'(0)=b$ and $V'(t)=-\gamma'(t)$, and that we have
	\[
	\frac{d}{dt}w(t)=-2\left\langle D_{t}\log_{\gamma(t)}(x),V(t)\right\rangle +2\left\langle D_{t}V(t),V(t)\right\rangle .
	\]
	Consider the smooth function $f_0:p\mapsto\frac{1}{2}\left\Vert \log_p(x)\right\Vert^2$. Because $\grad f_0(p)=-\log_p(x)$, we have $\Hess f_0(\gamma(t))[w]=\nabla_wX$, where $X:p\mapsto-\log_p(x)$ \citep[Section~4]{alimisis2020continuous}.  By \cref{prop:alimisis}, we have $\delta\left\Vert w\right\Vert ^{2}\leq\left\langle \Hess f_0(\gamma(t))[w],w\right\rangle \leq\zeta\left\Vert w\right\Vert ^{2}\leq\xi\left\Vert w\right\Vert ^{2}$ \citep[Appendix~D]{alimisis2021momentum}. Thus,
	\[
	-\frac{\xi-\delta}{2}\left\Vert w\right\Vert ^{2}=\delta\left\Vert w\right\Vert ^{2}-\frac{\xi+\delta}{2}\left\Vert w\right\Vert ^{2}\leq\left\langle \Hess f_0(\gamma(t))[w]-\frac{\xi+\delta}{2}w,w\right\rangle \leq\xi\left\Vert w\right\Vert ^{2}-\frac{\xi+\delta}{2}\left\Vert w\right\Vert ^{2}=\frac{\xi-\delta}{2}\left\Vert w\right\Vert ^{2}.
	\]
	Because $\Hess f_0(\gamma(t))$ is self-adjoint, it is diagonalizable. Thus,
	the norm of the operator $\Hess f_0(\gamma(t))-\frac{\xi+\delta}{2}I$ on $T_{\gamma(t)}M$ can be bounded as
	\[
	\left\Vert \Hess f_0(\gamma(t))-\frac{\xi+\delta}{2}I\right\Vert \leq\frac{\xi-\delta}{2}.
	\]
	Now, we have
	\begin{align*}
		-2\left\langle D_{t}\log_{\gamma(t)}(x),V(t)\right\rangle  & =2\left\langle \nabla_{\gamma'(t)}X,V(t)\right\rangle \\
		& =2\left\langle \Hess f_0(\gamma(t))[\gamma'(t)],V(t)\right\rangle \\
		& =2\left\langle \left(\Hess f_0(\gamma(t))-\frac{\xi+\delta}{2}I\right)\left(\gamma'(t)\right),V(t)\right\rangle +2\left\langle \frac{\xi+\delta}{2}\gamma'(t),V(t)\right\rangle \\
		& \leq2\left\Vert \left(\Hess f_0(\gamma(t))-\frac{\xi+\delta}{2}I\right)\left(\gamma'(t)\right)\right\Vert \left\Vert V(t)\right\Vert +2\left\langle \frac{\xi+\delta}{2}\gamma'(t),V(t)\right\rangle \\
		& \leq2\left\Vert \Hess f_0(\gamma(t))-\frac{\xi+\delta}{2}I\right\Vert \left\Vert \gamma'(t)\right\Vert \left\Vert V(t)\right\Vert +2\left\langle \frac{\xi+\delta}{2}\gamma'(t),V(t)\right\rangle \\
		& \leq2\frac{\xi-\delta}{2}\left\Vert \gamma'(t)\right\Vert \left\Vert V(t)\right\Vert +2\left\langle \frac{\xi+\delta}{2}\gamma'(t),V(t)\right\rangle.
	\end{align*}
	Because the parallel transport preserves inner product and norm, we obtain
	\begin{align*}
		-2\left\langle D_{t}\log_{\gamma(t)}(x),V(t)\right\rangle & \leq2\frac{\xi-\delta}{2}\left\Vert b\right\Vert \left\Vert a+(1-t)b\right\Vert +(\xi+\delta)\left\langle b,a+(1-t)b\right\rangle \\
		& =\frac{\xi-\delta}{2}\frac{1}{1-t}2\left\Vert (1-t)b\right\Vert \left\Vert a+(1-t)b\right\Vert +(\xi+\delta)\left\langle b,a+(1-t)b\right\rangle \\
		& \leq\frac{\xi-\delta}{2}\frac{1}{1-t}\left(\left\Vert (1-t)b\right\Vert ^{2}+\left\Vert a+(1-t)b\right\Vert ^{2}\right)+(\xi+\delta)\left\langle b,a+(1-t)b\right\rangle \\
		& =\frac{\xi-\delta}{2}\frac{1}{1-t}\left\Vert a\right\Vert ^{2}-2\xi\left\langle -b,a+(1-t)b\right\rangle \\
		& =\frac{\xi-\delta}{2}\frac{1}{1-t}\left\Vert a\right\Vert ^{2}-2\xi\left\langle D_{t}V(t),V(t)\right\rangle.
	\end{align*}
	Thus, for $t\in (0,r)$
	\[
	\frac{d}{dt}w(t)\leq\frac{\xi-\delta}{2}\frac{1}{1-r}\left\Vert a\right\Vert ^{2}-2(\xi-1)\left\langle D_{t}V(t),V(t)\right\rangle .
	\]
	Integrating both sides from $0$ to $r$, the result follows.
\end{proof}

\section{Convergence Analysis for RGD}
\label{app:rgd}

In this section, we review the iteration complexity of \ref{eq:rgd} with the fixed step size $\gamma_k=s$ under the assumptions in \cref{sec:assumption}. The results in this section correspond to \citep[Theorems~13 and 15]{zhang2016first}.

\subsection{Geodesically convex case}

We define the potential function as
\[
\phi_{k}=s(k+\zeta-1)\left(f\left(x_{k}\right)-f\left(x^{*}\right)\right)+\frac{1}{2}\left\Vert \log_{x_{k}}\left(x^{*}\right)\right\Vert ^{2}.
\]
The following theorem says that $\phi_k$ is decreasing.

\begin{theorem}
	\label{thm:rgd-c}
	Let $f$ be a geodesically convex and geodesically $L$-smooth function. If $s\leq\frac{1}{L}$, then the iterates of \ref{eq:rgd} satisfy
	\[
	s(k+\zeta)\left(f\left(x_{k+1}\right)-f\left(x^{*}\right)\right)+\frac{1}{2}\left\Vert \log_{x_{k+1}}\left(x^{*}\right)\right\Vert ^{2}\leq s(k+\zeta-1)\left(f\left(x_{k}\right)-f\left(x^{*}\right)\right)+\frac{1}{2}\left\Vert \log_{x_{k}}\left(x^{*}\right)\right\Vert ^{2}
	\]
	for all $k\geq0$.
\end{theorem}

\begin{proof}
	(Step 1).
	In this step, $\langle\cdot,\cdot\rangle$ and $\Vert\cdot\Vert$ always denote the inner product and the norm on $T_{x_k}M$. It follows from the geodesic convexity of $f$ that
	\begin{align*}
		f\left(x^{*}\right) & \geq f\left(x_{k}\right)+\left\langle \grad f\left(x_{k}\right),\log_{x_{k}}\left(x^{*}\right)\right\rangle \\
		& =f\left(x_{k}\right)-\frac{1}{s}\left\langle \log_{x_{k}}\left(x_{k+1}\right),\log_{x_{k}}\left(x^{*}\right)\right\rangle .
	\end{align*}
	By the geodesic $\frac{1}{s}$-smoothness of $f$, we have
	\begin{align*}
		f\left(x_{k+1}\right) & \leq f\left(x_{k}\right)+\left\langle \grad f\left(x_{k}\right),\log_{x_{k}}\left(x_{k+1}\right)\right\rangle +\frac{1}{2s}\left\Vert \log_{x_{k}}\left(x_{k+1}\right)\right\Vert ^{2}\\
		& =f\left(x_{k}\right)-\frac{1}{2s}\left\Vert \log_{x_{k}}\left(x_{k+1}\right)\right\Vert ^{2}.
	\end{align*}
	Taking a weighted sum of these inequalities yields
	\begin{align*}
		0 & \geq\left[f\left(x_{k}\right)-f\left(x^{*}\right)-\frac{1}{s}\left\langle \log_{x_{k}}\left(x_{k+1}\right),\log_{x_{k}}\left(x^{*}\right)\right\rangle \right]\\
		& \quad+(k+\zeta)\left[f\left(x_{k+1}\right)-f\left(x_{k}\right)+\frac{1}{2s}\left\Vert \log_{x_{k}}\left(x_{k+1}\right)\right\Vert ^{2}\right]\\
		& =(k+\zeta)\left(f\left(x_{k+1}\right)-f\left(x^{*}\right)\right)-(k+\zeta-1)\left(f\left(x_{k}\right)-f\left(x^{*}\right)\right)\\
		& \quad-\frac{1}{s}\left\langle \log_{x_{k}}\left(x_{k+1}\right),\log_{x_{k}}\left(x^{*}\right)\right\rangle +\frac{k+\zeta}{2s}\left\Vert \log_{x_{k}}\left(x_{k+1}\right)\right\Vert ^{2}\\
		& \geq(k+\zeta)\left(f\left(x_{k+1}\right)-f\left(x^{*}\right)\right)-(k+\zeta-1)\left(f\left(x_{k}\right)-f\left(x^{*}\right)\right)\\
		& \quad-\frac{1}{s}\left\langle \log_{x_{k}}\left(x_{k+1}\right),\log_{x_{k}}\left(x^{*}\right)\right\rangle +\frac{\zeta}{2s}\left\Vert \log_{x_{k}}\left(x_{k+1}\right)\right\Vert ^{2}\\
		& =(k+\zeta)\left(f\left(x_{k+1}\right)-f\left(x^{*}\right)\right)-(k+\zeta-1)\left(f\left(x_{k}\right)-f\left(x^{*}\right)\right)\\
		& \quad+\frac{1}{2s}\left(\zeta\left\Vert \log_{x_{k}}\left(x_{k+1}\right)\right\Vert ^{2}-2\left\langle \log_{x_{k}}\left(x_{k+1}\right),\log_{x_{k}}\left(x^{*}\right)\right\rangle \right)\\
		& =(k+\zeta)\left(f\left(x_{k+1}\right)-f\left(x^{*}\right)\right)-(k+\zeta-1)\left(f\left(x_{k}\right)-f\left(x^{*}\right)\right)\\
		& \quad+\frac{1}{2s}\left(\left\Vert \log_{x_{k}}\left(x_{k+1}\right)-\log_{x_{k}}\left(x^{*}\right)\right\Vert ^{2}+(\zeta-1)\left\Vert \log_{x_{k}}\left(x_{k+1}\right)\right\Vert ^{2}-\left\Vert \log_{x_{k}}\left(x^{*}\right)\right\Vert ^{2}\right).
	\end{align*}
	
	{(Step 2: Handling metric distortion).}
	By \cref{lem:distortion1} with $p_A=x_k$, $p_B=x_{k+1}$, $x=x^{*}$, $v_A=\log_{x_k}\left(x_{k+1}\right)$, $v_B=0$, $r=1$, we have
	\[
	\left\Vert \log_{x_{k+1}}\left(x^{*}\right)\right\Vert _{x_{k+1}}^{2}\leq\left\Vert \log_{x_{k}}\left(x_{k+1}\right)-\log_{x_{k}}\left(x^{*}\right)\right\Vert _{x_{k}}^{2}+(\zeta-1)\left\Vert \log_{x_{k}}\left(x_{k+1}\right)\right\Vert _{x_{k}}^{2}.
	\]
	Combining this inequality with the result in Step 1 gives
	\begin{align*}
		0 & \geq(k+\zeta)\left(f\left(x_{k+1}\right)-f\left(x^{*}\right)\right)-(k+\zeta-1)\left(f\left(x_{k}\right)-f\left(x^{*}\right)\right)\\
		& \quad+\frac{1}{2s}\left(\left\Vert \log_{x_{k}}\left(x_{k+1}\right)-\log_{x_{k}}\left(x^{*}\right)\right\Vert _{x_{k}}^{2}+(\zeta-1)\left\Vert \log_{x_{k}}\left(x_{k+1}\right)\right\Vert _{x_{k}}^{2}-\left\Vert \log_{x_{k}}\left(x^{*}\right)\right\Vert _{x_{k}}^{2}\right)\\
		& \quad+\frac{1}{2s}\left(\left\Vert \log_{x_{k+1}}\left(x^{*}\right)\right\Vert _{x_{k+1}}^{2}-\left\Vert \log_{x_{k}}\left(x_{k+1}\right)-\log_{x_{k}}\left(x^{*}\right)\right\Vert _{x_{k}}^{2}-(\zeta-1)\left\Vert \log_{x_{k}}\left(x_{k+1}\right)\right\Vert _{x_{k}}^{2}\right)\\
		& =(k+\zeta)\left(f\left(x_{k+1}\right)-f\left(x^{*}\right)\right)-(k+\zeta-1)\left(f\left(x_{k}\right)-f\left(x^{*}\right)\right)+\frac{1}{2s}\left\Vert \log_{x_{k+1}}\left(x^{*}\right)\right\Vert _{x_{k+1}}^{2}-\frac{1}{2s}\left\Vert \log_{x_{k}}\left(x^{*}\right)\right\Vert _{x_{k}}^{2}\\
		& =\frac{\phi_{k+1}-\phi_{k}}{s}.
	\end{align*}
\end{proof}

\begin{corollary}
	\label{cor:rgd-c}
	Let $f$ be a geodesically convex and geodesically $L$-smooth function. Then, \ref{eq:rgd} with the step size $s=\frac{1}{L}$ finds an $\epsilon$-approximate solution in $O\left(\frac{\zeta L}{\epsilon}\right)$ iterations.
\end{corollary}

\begin{proof}
	It follows from \cref{thm:rgd-c} that
	\[
	f\left(x_{k}\right)-f\left(x^{*}\right)\leq\frac{\phi_{k}}{s(k+\zeta-1)}\leq\frac{\phi_{0}}{s(k+\zeta-1)}=\frac{1}{s(k+\zeta-1)}\left(s(\zeta-1)\left(f\left(x_{0}\right)-f\left(x^{*}\right)\right)+\frac{1}{2}\left\Vert \log_{x_{0}}\left(x^{*}\right)\right\Vert ^{2}\right).
	\]
	By geodesic $\frac{1}{s}$-smoothness of $f$, we have
	\[
	f\left(x_{k}\right)-f\left(x^{*}\right)\le\frac{1}{s(k+\zeta-1)}\left(s(\zeta-1)\frac{1}{2s}\left\Vert \log_{x_{0}}\left(x^{*}\right)\right\Vert ^{2}+\frac{1}{2}\left\Vert \log_{x_{0}}\left(x^{*}\right)\right\Vert ^{2}\right)=\frac{\zeta L}{2(k+\zeta-1)}\left\Vert \log_{x_{0}}\left(x^{*}\right)\right\Vert ^{2}.
	\]
	Thus, we have $f\left(x_k\right)-f\left(x^{*}\right)\leq\epsilon$ whenever $k\geq\frac{\zeta L}{2\epsilon}\left\Vert \log_{x_{0}}\left(x^{*}\right)\right\Vert ^{2}-(\zeta-1)$. Thus we obtain an $O\left(\frac{\zeta L}{\epsilon}\right)$ iteration complexity.
\end{proof}

This result implies that the iteration complexity of \ref{eq:rgd} for geodesically convex case is the same as that of \ref{eq:gd}, since $\zeta$ is a constant.

\subsection{Geodesically strongly convex case}

We define the potential function as
\[
\phi_{k}=(1-\mu s)^{-k}\left(f\left(x_{k}\right)-f\left(x^{*}\right)+\frac{\mu}{2}\left\Vert \log_{x_{k}}\left(x^{*}\right)\right\Vert ^{2}\right).
\]
The following theorem states that $\phi_k$ is decreasing.

\begin{restatable}{theorem}{rgdscthm}
	\label{thm:rgd-sc}
	Let $f$ be a geodesically $\mu$-strongly convex and geodesically $L$-smooth function. If $s\leq\min\left\{ \frac{1}{L},\frac{1}{\zeta\mu}\right\} $, then the iterates of \ref{eq:rgd} satisfy
	\[
	(1-\mu s)^{-(k+1)}\left(f\left(x_{k+1}\right)-f\left(x^{*}\right)+\frac{\mu}{2}\left\Vert \log_{x_{k+1}}\left(x^{*}\right)\right\Vert ^{2}\right)\leq(1-\mu s)^{-k}\left(f\left(x_{k}\right)-f\left(x^{*}\right)+\frac{\mu}{2}\left\Vert \log_{x_{k}}\left(x^{*}\right)\right\Vert ^{2}\right)
	\]
	for all $k\geq0$.
\end{restatable}

\begin{proof}
	(Step 1).
	In this step, $\langle\cdot,\cdot\rangle$ and $\Vert\cdot\Vert$ always denote the inner product and the norm on $T_{x_k}M$. Set $q=\mu s$. By geodesic $\mu$-strong convexity of $f$, we have
	\begin{align*}
		f\left(x^{*}\right) & \geq f\left(x_{k}\right)+\left\langle \grad f\left(x_{k}\right),\log_{x_{k}}\left(x^{*}\right)\right\rangle +\frac{\mu}{2}\left\Vert \log_{x_{k}}\left(x^{*}\right)\right\Vert ^{2}\\
		& =f\left(x_{k}\right)-\frac{1}{s}\left\langle \log_{x_{k}}\left(x_{k+1}\right),\log_{x_{k}}\left(x^{*}\right)\right\rangle +\frac{q}{2s}\left\Vert \log_{x_{k}}\left(x^{*}\right)\right\Vert ^{2}.
	\end{align*}
	By geodesic $\frac{1}{s}$-smoothness of $f$, we have
	\begin{align*}
		f\left(x_{k+1}\right) & \leq f\left(x_{k}\right)+\left\langle \grad f\left(x_{k}\right),\log_{x_{k}}\left(x_{k+1}\right)\right\rangle +\frac{1}{2s}\left\Vert \log_{x_{k}}\left(x_{k+1}\right)\right\Vert ^{2}\\
		& =f\left(x_{k}\right)-\frac{1}{2s}\left\Vert \log_{x_{k}}\left(x_{k+1}\right)\right\Vert ^{2}.
	\end{align*}
	Note that $\zeta q\leq1$. Taking weighted sum of these inequalities, we arrive to the valid
	inequality
	\begin{align*}
		0 & \geq q\left[f\left(x_{k}\right)-f\left(x^{*}\right)-\frac{1}{s}\left\langle \log_{x_{k}}\left(x_{k+1}\right),\log_{x_{k}}\left(x^{*}\right)\right\rangle +\frac{q}{2s}\left\Vert \log_{x_{k}}\left(x^{*}\right)\right\Vert ^{2}.\right]\\
		& \quad+\left[f\left(x_{k+1}\right)-f\left(x_{k}\right)+\frac{1}{2s}\left\Vert \log_{x_{k}}\left(x_{k+1}\right)\right\Vert ^{2}\right]\\
		& =f\left(x_{k+1}\right)-f\left(x^{*}\right)-(1-q)\left(f\left(x_{k}\right)-f\left(x^{*}\right)\right)\\
		& \quad-\frac{q}{s}\left\langle \log_{x_{k}}\left(x_{k+1}\right),\log_{x_{k}}\left(x^{*}\right)\right\rangle +\frac{q^{2}}{2s}\left\Vert \log_{x_{k}}\left(x^{*}\right)\right\Vert ^{2}+\frac{1}{2s}\left\Vert \log_{x_{k}}\left(x_{k+1}\right)\right\Vert ^{2}\\
		& \geq f\left(x_{k+1}\right)-f\left(x^{*}\right)-(1-q)\left(f\left(x_{k}\right)-f\left(x^{*}\right)\right)\\
		& \quad+\frac{q}{2s}\left(-2\left\langle \log_{x_{k}}\left(x_{k+1}\right),\log_{x_{k}}\left(x^{*}\right)\right\rangle +q\left\Vert \log_{x_{k}}\left(x^{*}\right)\right\Vert ^{2}+\zeta\left\Vert \log_{x_{k}}\left(x_{k+1}\right)\right\Vert ^{2}\right)\\
		& =f\left(x_{k+1}\right)-f\left(x^{*}\right)-(1-q)\left(f\left(x_{k}\right)-f\left(x^{*}\right)\right)\\
		& \quad+\frac{q}{2s}\left(\left\Vert \log_{x_{k}}\left(x_{k+1}\right)-\log_{x_{k}}\left(x^{*}\right)\right\Vert ^{2}+(\zeta-1)\left\Vert \log_{x_{k}}\left(x_{k+1}\right)\right\Vert ^{2}-(1-q)\left\Vert \log_{x_{k}}\left(x^{*}\right)\right\Vert ^{2}\right).
	\end{align*}
	
	(Step 2: Handle metric distortion).
	By \cref{lem:distortion1} with $p_A=x_k$, $p_B=x_{k+1}$, $x=x^{*}$, $v_A=\log_{x_k}\left(x_{k+1}\right)$, $v_B=0$, $r=1$, we have
	\[
	\left\Vert \log_{x_{k+1}}\left(x^{*}\right)\right\Vert _{x_{k+1}}^{2}\leq\left\Vert \log_{x_{k}}\left(x_{k+1}\right)-\log_{x_{k}}\left(x^{*}\right)\right\Vert _{x_{k}}^{2}+(\zeta-1)\left\Vert \log_{x_{k}}\left(x_{k+1}\right)\right\Vert _{x_{k}}^{2}.
	\]
	Combining this inequality with the result in Step 1 gives
	\begin{align*}
		0 & \geq f\left(x_{k+1}\right)-f\left(x^{*}\right)-(1-q)\left(f\left(x_{k}\right)-f\left(x^{*}\right)\right)\\
		& \quad+\frac{q}{2s}\left(\left\Vert \log_{x_{k}}\left(x_{k+1}\right)-\log_{x_{k}}\left(x^{*}\right)\right\Vert _{x_{k}}^{2}+(\zeta-1)\left\Vert \log_{x_{k}}\left(x_{k+1}\right)\right\Vert _{x_{k}}^{2}-(1-q)\left\Vert \log_{x_{k}}\left(x^{*}\right)\right\Vert _{x_{k}}^{2}\right)\\
		& \quad+\frac{q}{2s}\left(\left\Vert \log_{x_{k+1}}\left(x^{*}\right)\right\Vert _{x_{k+1}}^{2}-\left\Vert \log_{x_{k}}\left(x_{k+1}\right)-\log_{x_{k}}\left(x^{*}\right)\right\Vert _{x_{k}}^{2}-(\zeta-1)\left\Vert \log_{x_{k}}\left(x_{k+1}\right)\right\Vert _{x_{k}}^{2}\right)\\
		0 & =f\left(x_{k+1}\right)-f\left(x^{*}\right)-(1-q)\left(f\left(x_{k}\right)-f\left(x^{*}\right)\right)+\frac{q}{2s}\left\Vert \log_{x_{k+{\color{red}1}}}\left(x^{*}\right)\right\Vert _{x_{k+1}}^{2}-\frac{q}{2s}(1-q)\left\Vert \log_{x_{k}}\left(x^{*}\right)\right\Vert _{x_{k}}^{2}\\
		& =\left(f\left(x_{k+1}\right)-f\left(x^{*}\right)+\frac{\mu}{2}\left\Vert \log_{x_{k+1}}\left(x^{*}\right)\right\Vert _{x_{k+1}}^{2}\right)-(1-q)\left(f\left(x_{k}\right)-f\left(x^{*}\right)+\frac{\mu}{2}\left\Vert \log_{x_{k}}\left(x^{*}\right)\right\Vert _{x_{k}}^{2}\right)\\
		& =(1-q)^{(k+1)}\left(\phi_{k+1}-\phi_{k}\right).
	\end{align*}
\end{proof}

\begin{corollary}
	\label{cor:rgd-sc}
	Let $f$ be a geodesically $\mu$-strongly convex and geodesically $L$-smooth function. Then, \ref{eq:rgd} with step size $s=\frac{1}{\zeta L}$ finds an $\epsilon$-approximate solution in $O\left(\frac{\zeta L}{\mu}\log\frac{L}{\epsilon}\right)$ iterations.
\end{corollary}

\begin{proof}
	By \cref{thm:rgd-sc}, we have
	\[
	f\left(x_{k}\right)-f\left(x^{*}\right)\leq\left(1-\mu s\right)^{k}\phi_{k}\leq\left(1-\mu s\right)^{k}\phi_{0}=\left(1-\mu s\right)^{k}\left(f\left(x_{0}\right)-f\left(x^{*}\right)+\frac{\mu}{2}\left\Vert \log_{x_{0}}\left(x^{*}\right)\right\Vert ^{2}\right).
	\]
	It follows from the geodesic $L$-smoothness of $f$ and the inequality $\left(1-\frac{\mu}{\zeta L}\right)^{k}\leq e^{-\frac{\mu}{\zeta L}k}$ that
	\[
	f\left(x_{k}\right)-f\left(x^{*}\right)\leq\left(1-\frac{\mu}{\zeta L}\right)^{k}\left(\frac{L}{2}\left\Vert \log_{x_{0}}\left(x^{*}\right)\right\Vert ^{2}+\frac{\mu}{2}\left\Vert \log_{x_{0}}\left(x^{*}\right)\right\Vert ^{2}\right)\leq e^{-\frac{\mu}{\zeta L}k}L\left\Vert \log_{x_{0}}\left(x^{*}\right)\right\Vert ^{2}.
	\]
	Thus, we have $f\left(x_k\right)-f\left(x^{*}\right)\leq\epsilon$ whenever $k\geq\frac{\zeta L}{\mu}\log\left(\frac{L}{\epsilon}\left\Vert \log_{x_{0}}\left(x^{*}\right)\right\Vert ^{2}\right)$. Accordingly, we obtain an $O\left(\frac{\zeta L}{\mu}\log\frac{L}{\epsilon}\right)$ iteration complexity.
\end{proof}

This result implies that the iteration complexity of \ref{eq:rgd} for g-strongly convex case is the same as that of \ref{eq:gd}, since $\zeta$ is a constant. Another proof of the iteration complexity of \ref{eq:rgd} for g-strongly convex functions can be found in \citep[Proposition~1.8]{criscitiello2021negative}.

\section{Convergence Analysis for RNAG-C}
\label{app:rnag-c}

\rnagcthm*

\begin{proof}
	(Step 1).
	In this step, $\langle\cdot,\cdot\rangle$ and $\Vert\cdot\Vert$ always denote the inner product and the norm on $T_{y_k}M$. It is easy to check that $\grad f\left(y_{k}\right)=-\frac{\xi}{s\lambda_{k}}\left(\bar{\bar{v}}_{k+1}-v_{k}\right)$,
	$\log_{y_{k}}\left(x_{k}\right)=-\frac{\xi}{\lambda_{k}-1}v_{k}$,\footnote{Note that $y_{k}=\exp_{x_{k}}\left(\frac{\xi}{\lambda_k+(\xi-1)}\bar{v}_{k}\right)$
		and $v_{k}=\Gamma_{x_{k}}^{y_{k}}\left(\bar{v}_{k}-\log_{x_{k}}\left(y_{k}\right)\right)=\Gamma_{x_{k}}^{y_{k}}\left(\left(1-\frac{\xi}{\lambda_k+(\xi-1)}\right)\bar{v}_{k}\right)$.
		Let $\gamma_{1}$ be the geodesic such that $\gamma_{1}(0)=x_{k}$
		and $\gamma_{1}(1)=y_{k}$, then $\gamma_{1}'(0)=\log_{x_{k}}\left(y_{k}\right)$.
		Let $\gamma_{2}$ be the geodesic defined as $\gamma_{2}(t)=\gamma_{1}(1-t)$.
		Then, $\log_{y_{k}}\left(x_{k}\right)=\gamma_{2}'(0)=-\gamma_{1}'(1)=-\Gamma_{x_{k}}^{y_{k}}\left(\gamma_{1}'(0)\right)=-\Gamma_{x_{k}}^{y_{k}}\left(\log_{x_{k}}\left(y_{k}\right)\right)$.
		Now, we have $\log_{y_{k}}\left(x_{k}\right)=-\Gamma_{x_{k}}^{y_{k}}\left(\log_{x_{k}}\left(y_{k}\right)\right)=-\frac{\xi}{\lambda_{k}+(\xi-1)}\Gamma_{x_{k}}^{y_{k}}\left(\bar{v}_{k}\right)=-\frac{\frac{\xi}{\lambda_{k}+(\xi-1)}}{1-\frac{\xi}{\lambda_{k}+(\xi-1)}}v_{k}=-\frac{\xi}{\lambda_{k}-1}v_{k}$.} and $\lambda_{k}^{2}-\lambda_{k}\leq\lambda_{k-1}^{2}$. By the geodesic
	convexity of $f$, we have
	\begin{align*}
		f\left(x^{*}\right) & \geq f\left(y_{k}\right)+\left\langle \grad f\left(y_{k}\right),\log_{y_{k}}\left(x^{*}\right)\right\rangle \\
		& =f\left(y_{k}\right)-\frac{\xi}{s\lambda_{k}}\left\langle \bar{\bar{v}}_{k+1}-v_{k},\log_{y_{k}}\left(x^{*}\right)\right\rangle ,\\
		f\left(x_{k}\right) & \geq f\left(y_{k}\right)+\left\langle \grad f\left(y_{k}\right),\log_{y_{k}}\left(x_{k}\right)\right\rangle \\
		& =f\left(y_{k}\right)+\frac{\xi^{2}}{s\left(\lambda_{k}^{2}-\lambda_{k}\right)}\left\langle \bar{\bar{v}}_{k+1}-v_{k},v_{k}\right\rangle .
	\end{align*}
	It follows from the geodesic $\frac{1}{s}$-smoothness of $f$ that
	\begin{align*}
		f\left(x_{k+1}\right) & \leq f\left(y_{k}\right)+\left\langle \grad f\left(y_{k}\right),\log_{y_{k}}\left(x_{k+1}\right)\right\rangle +\frac{1}{2s}\left\Vert \log_{y_{k}}\left(x_{k+1}\right)\right\Vert ^{2}\\
		& =f\left(y_{k}\right)-\frac{s}{2}\left\Vert \grad f\left(y_k\right)\right\Vert^2\\
		& =f\left(y_{k}\right)-\frac{\xi^{2}}{2s\lambda_{k}^{2}}\left\Vert \bar{\bar{v}}_{k+1}-v_{k}\right\Vert ^{2}.
	\end{align*}
	
	Taking a weighted sum of these inequalities yields
	\begin{align*}
		0 & \geq\lambda_{k}\left[f\left(y_{k}\right)-f\left(x^{*}\right)-\frac{\xi}{s\lambda_{k}}\left\langle \bar{\bar{v}}_{k+1}-v_{k},\log_{y_{k}}\left(x^{*}\right)\right\rangle \right]\\
		& \quad+\left(\lambda_{k}^{2}-\lambda_{k}\right)\left[f\left(y_{k}\right)-f\left(x_{k}\right)+\frac{\xi^{2}}{s\left(\lambda_{k}^{2}-\lambda_{k}\right)}\left\langle \bar{\bar{v}}_{k+1}-v_{k},v_{k}\right\rangle \right]\\
		& \quad+\lambda_{k}^{2}\left[f\left(x_{k+1}\right)-f\left(y_{k}\right)+\frac{\xi^{2}}{2s\lambda_{k}^{2}}\left\Vert \bar{\bar{v}}_{k+1}-v_{k}\right\Vert ^{2}\right]\\
		& =\lambda_{k}^{2}\left(f\left(x_{k+1}\right)-f\left(x^{*}\right)\right)-\left(\lambda_{k}^{2}-\lambda_{k}\right)\left(f\left(x_{k}\right)-f\left(x^{*}\right)\right)\\
		& \quad-\frac{\xi}{s}\left\langle \bar{\bar{v}}_{k+1}-v_{k},\log_{y_{k}}\left(x^{*}\right)\right\rangle +\frac{\xi^{2}}{s}\left\langle \bar{\bar{v}}_{k+1}-v_{k},v_{k}\right\rangle +\frac{\xi^{2}}{2s}\left\Vert \bar{\bar{v}}_{k+1}-v_{k}\right\Vert ^{2}\\
		& \geq\lambda_{k}^{2}\left(f\left(x_{k+1}\right)-f\left(x^{*}\right)\right)-\lambda_{k-1}^{2}\left(f\left(x_{k}\right)-f\left(x^{*}\right)\right)\\
		& \quad+\frac{\xi}{2s}\left(-2\left\langle \bar{\bar{v}}_{k+1}-v_{k},\log_{y_{k}}\left(x^{*}\right)\right\rangle +2\xi\left\langle \bar{\bar{v}}_{k+1}-v_{k},v_{k}\right\rangle +\xi\left\Vert \bar{\bar{v}}_{k+1}-v_{k}\right\Vert ^{2}\right)\\
		& =\lambda_{k}^{2}\left(f\left(x_{k+1}\right)-f\left(x^{*}\right)\right)-\lambda_{k-1}^{2}\left(f\left(x_{k}\right)-f\left(x^{*}\right)\right)\\
		& \quad+\frac{\xi}{2s}\left(\left\Vert \bar{\bar{v}}_{k+1}-v_{k}\right\Vert ^{2}-2\left\langle \bar{\bar{v}}_{k+1}-v_{k},\log_{y_{k}}\left(x^{*}\right)-v_{k}\right\rangle +2(\xi-1)\left\langle \bar{\bar{v}}_{k+1}-v_{k},v_{k}\right\rangle +(\xi-1)\left\Vert \bar{\bar{v}}_{k+1}-v_{k}\right\Vert ^{2}\right).\\
		& =\lambda_{k}^{2}\left(f\left(x_{k+1}\right)-f\left(x^{*}\right)\right)-\lambda_{k-1}^{2}\left(f\left(x_{k}\right)-f\left(x^{*}\right)\right)\\
		& \quad+\frac{\xi}{2s}\left(\left\Vert \bar{\bar{v}}_{k+1}-v_{k}\right\Vert ^{2}-2\left\langle \bar{\bar{v}}_{k+1}-v_{k},\log_{y_{k}}\left(x^{*}\right)-v_{k}\right\rangle +(\xi-1)\left\Vert \bar{\bar{v}}_{k+1}\right\Vert ^{2}-(\xi-1)\left\Vert v_{k}\right\Vert ^{2}\right).
	\end{align*}
	Note that
	\[
	\left\Vert \bar{\bar{v}}_{k+1}-\log_{y_{k}}\left(x^{*}\right)\right\Vert ^{2}-\left\Vert v_{k}-\log_{y_{k}}\left(x^{*}\right)\right\Vert ^{2}=\left\Vert \bar{\bar{v}}_{k+1}-v_{k}\right\Vert ^{2}-2\left\langle \bar{\bar{v}}_{k+1}-v_{k},\log_{y_{k}}\left(x^{*}\right)-v_{k}\right\rangle.
	\]
	Thus, we obtain
	\begin{align*}
		0 & \geq\lambda_{k}^{2}\left(f\left(x_{k+1}\right)-f\left(x^{*}\right)\right)-\lambda_{k-1}^{2}\left(f\left(x_{k}\right)-f\left(x^{*}\right)\right)\\
		& \quad+\frac{\xi}{2s}\left(\left\Vert \bar{\bar{v}}_{k+1}-\log_{y_{k}}\left(x^{*}\right)\right\Vert ^{2}-\left\Vert v_{k}-\log_{y_{k}}\left(x^{*}\right)\right\Vert ^{2}+(\xi-1)\left\Vert \bar{\bar{v}}_{k+1}\right\Vert ^{2}-(\xi-1)\left\Vert v_{k}\right\Vert ^{2}\right).
	\end{align*}
	
	(Step 2: Handle metric distortion).
	By \cref{lem:distortion2} with $p_A=y_k$, $p_B=x_{k+1}$, $x=x^{*}$, $v_A=\bar{\bar{v}}_{k+1}$, $v_B=\bar{v}_{k+1}$, $a=v_k$, $b=-\gamma_k \grad f\left(y_k\right)=-\frac{s\lambda_k}{\xi}\grad f\left(y_k\right)$, $r=\frac{s}{\gamma_k}=\frac{\xi}{\lambda_k}\in(0,1)$, we have
	\begin{align*}
		& \left\Vert \log_{x_{k+1}}\left(x^{*}\right)-\bar{v}_{k+1}\right\Vert _{x_{k+1}}^{2}+(\xi-1)\left\Vert \bar{v}_{k+1}\right\Vert _{x_{k+1}}^{2}\\
		& \leq\left\Vert \log_{y_{k}}\left(x^{*}\right)-\bar{\bar{v}}_{k+1}\right\Vert _{y_{k}}^{2}+(\xi-1)\left\Vert \bar{\bar{v}}_{k+1}\right\Vert _{y_{k}}^{2}+\frac{\xi-\delta}{2}\left(\frac{1}{1-\xi/\lambda_k}-1\right)\left\Vert v_{k}\right\Vert _{y_{k}}^{2}.
	\end{align*}
	It follows from \cref{lem:distortion1} with $p_A=x_k$, $p_B=y_k$, $x=x^{*}$, $v_A=\bar{v}_{k}$, $v_B=v_{k}$, $r=\tau_k=\frac{\xi}{\lambda_k+\xi-1}$ that
	\begin{align*}
		\left\Vert \log_{x_{k}}\left(x^{*}\right)-\bar{v}_{k}\right\Vert _{x_{k}}^{2}+(\xi-1)\left\Vert \bar{v}_{k}\right\Vert _{x_{k}}^{2} & =\left(\left\Vert \log_{x_{k}}\left(x^{*}\right)-\bar{v}_{k}\right\Vert _{x_{k}}^{2}+(\zeta-1)\left\Vert \bar{v}_{k}\right\Vert _{x_{k}}^{2}\right)+(\xi-\zeta)\left\Vert \bar{v}_{k}\right\Vert _{x_{k}}^{2}\\
		& \geq\left(\left\Vert \log_{y_{k}}\left(x^{*}\right)-v_{k}\right\Vert _{y_{k}}^{2}+(\zeta-1)\left\Vert v_{k}\right\Vert _{y_{k}}^{2}\right)+(\xi-\zeta)\left\Vert \bar{v}_{k}\right\Vert _{x_{k}}^{2}\\
		& =\left\Vert \log_{y_{k}}\left(x^{*}\right)-v_{k}\right\Vert _{y_{k}}^{2}+(\zeta-1)\left\Vert v_{k}\right\Vert _{y_{k}}^{2}+(\xi-\zeta)\frac{1}{\left(1-\tau_{k}\right)^{2}}\left\Vert v_{k}\right\Vert _{y_{k}}^{2}\\
		& =\left\Vert \log_{y_{k}}\left(x^{*}\right)-v_{k}\right\Vert _{y_{k}}^{2}+(\xi-1)\left\Vert v_{k}\right\Vert _{y_{k}}^{2}+(\xi-\zeta)\left(\frac{1}{\left(1-\tau_{k}\right)^{2}}-1\right)\left\Vert v_{k}\right\Vert _{y_{k}}^{2},
	\end{align*}
	Combining these inequalities with the result in Step~1 gives
	\begin{align*}
		0 & \geq s\lambda_{k}^{2}\left(f\left(x_{k+1}\right)-f\left(x^{*}\right)\right)-\lambda_{k-1}^{2}\left(f\left(x_{k}\right)-f\left(x^{*}\right)\right)\\
		& \quad+\frac{\xi}{2}\left(\left\Vert \bar{\bar{v}}_{k+1}-\log_{y_{k}}\left(x^{*}\right)\right\Vert ^{2}+(\xi-1)\left\Vert \bar{\bar{v}}_{k+1}\right\Vert ^{2}-\left\Vert v_{k}-\log_{y_{k}}\left(x^{*}\right)\right\Vert ^{2}-(\xi-1)\left\Vert v_{k}\right\Vert ^{2}\right).\\
		& \quad+\frac{\xi}{2}\left[\left\Vert \log_{x_{k+1}}\left(x^{*}\right)-\bar{v}_{k+1}\right\Vert _{x_{k+1}}^{2}+(\xi-1)\left\Vert \bar{v}_{k+1}\right\Vert _{x_{k+1}}^{2}\right.\\
		& \quad\quad\left.-\left\Vert \log_{y_{k}}\left(x^{*}\right)-\bar{\bar{v}}_{k+1}\right\Vert _{y_{k}}^{2}-(\xi-1)\left\Vert \bar{\bar{v}}_{k+1}\right\Vert _{y_{k}}^{2}-\frac{\xi-\delta}{2}\left(\frac{1}{1-\xi/\lambda_k}-1\right)\left\Vert v_{k}\right\Vert _{y_{k}}^{2}\right]\\
		& \quad+\frac{\xi}{2}\left[\left\Vert \log_{y_{k}}\left(x^{*}\right)-v_{k}\right\Vert _{y_{k}}^{2}+(\xi-1)\left\Vert v_{k}\right\Vert _{y_{k}}^{2}+(\xi-\zeta)\left(\frac{1}{\left(1-\tau_{k}\right)^{2}}-1\right)\left\Vert v_{k}\right\Vert _{y_{k}}^{2}\right.\\
		& \quad\quad\left.-\left\Vert \log_{x_{k}}\left(x^{*}\right)-\bar{v}_{k}\right\Vert _{x_{k}}^{2}-(\xi-1)\left\Vert \bar{v}_{k}\right\Vert _{x_{k}}^{2}\right]\\
		& =\phi_{k+1}-\phi_{k}+\frac{\xi}{2}\left((\xi-\zeta)\left(\frac{1}{\left(1-\tau_{k}\right)^{2}}-1\right)-\frac{\xi-\delta}{2}\left(\frac{1}{1-\xi/\lambda_k}-1\right)\right)\left\Vert v_{k}\right\Vert _{y_{k}}^{2}\\
		& \geq\phi_{k+1}-\phi_{k}.
	\end{align*}
\end{proof}

\rnagccor*

\begin{proof}
	(Step 1: Checking the condition for \cref{thm:rnag-c}).
	A straightforward calculation shows that
	\[
	2\left(\frac{1}{1-t}-1\right)\leq3\left(\frac{1}{1-(3/4)t}-1\right).
	\]
	for all $t\in(0,1/3]$. For convenience, let $r=\frac{s}{\gamma_{k}}=\frac{\xi}{\lambda_{k}}=\frac{2\xi}{k+6\xi}\in(0,1/3]$.
	Then, $\tau_{k}=\frac{\xi}{\lambda_{k}+(\xi-1)}=\frac{2\xi}{k+6\xi+2(\xi-1)}\geq\frac{2\xi}{k+8\xi}\geq\frac{2\xi}{\frac{4}{3}(k+6\xi)}=\frac{3}{4}r$. Now, we have
	\begin{align*}
		(\xi-\zeta)\left(\frac{1}{\left(1-\tau_{k}\right)^{2}}-1\right) & \geq(\xi-\zeta)\left(\frac{1}{1-\tau_{k}}-1\right)\\
		& \geq(\xi-\zeta)\left(\frac{1}{1-\frac{3}{4}r}-1\right)\\
		& \geq\frac{\xi-\delta}{2}\left(\frac{1}{1-r}-1\right).
	\end{align*}
	
	(Step 2: Computing iteration complexity).
	By \cref{thm:rnag-c}, we have
	\[
	f\left(x_{k}\right)-f\left(x^{*}\right)\leq\frac{\phi_{k}}{s\lambda_{k-1}^{2}}\leq\frac{\phi_{0}}{s\lambda_{k-1}^{2}}=\frac{1}{s\lambda_{k-1}^{2}}\left(s\lambda_{-1}^{2}\left(f\left(x_{0}\right)-f\left(x^{*}\right)\right)+\frac{\xi}{2}\left\Vert \log_{x_{0}}\left(x^{*}\right)\right\Vert ^{2}\right).
	\]
	It follows from the geodesic $\frac{1}{s}$-smoothness of $f$ that
	\begin{align*}
		f\left(x_{k}\right)-f\left(x^{*}\right) & \leq\frac{1}{s\lambda_{k-1}^{2}}\left(s\lambda_{-1}^{2}\frac{1}{2s}\left\Vert \log_{x_{0}}\left(x^{*}\right)\right\Vert ^{2}+\frac{\xi}{2}\left\Vert \log_{x_{0}}\left(x^{*}\right)\right\Vert ^{2}\right)\\
		& =\frac{1}{s\lambda_{k-1}^{2}}\left(\frac{\lambda_{-1}^{2}}{2}+\frac{\xi}{2}\right)\left\Vert \log_{x_{0}}\left(x^{*}\right)\right\Vert ^{2}\\
		& =\frac{4L}{(k-1+6\xi)^{2}}\left(\frac{(6\xi-1)^2}{8}+\frac{\xi}{2}\right)\left\Vert \log_{x_{0}}\left(x^{*}\right)\right\Vert ^{2}\\
		& \leq\frac{4L}{(k-1)^{2}}\left(\frac{(6\xi-1)^2}{8}+\frac{\xi}{2}\right)\left\Vert \log_{x_{0}}\left(x^{*}\right)\right\Vert ^{2}.
	\end{align*}
	Thus, we have $f\left(x_{k}\right)-f\left(x^{*}\right)\leq\epsilon$ whenever
	\[
	(k-1)^{2}\geq\frac{4L}{\epsilon}\left(\frac{(6\xi-1)^{2}}{8}+\frac{\xi}{2}\right)\left\Vert \log_{x_{0}}\left(x^{*}\right)\right\Vert ^{2}.
	\]
	This implies that RNAG-C has an $O\left(\xi\sqrt{\frac{L}{\epsilon}}\right)$ iteration complexity.
\end{proof}

\section{Convergence Analysis for RNAG-SC}
\label{app:rnag-sc}

\rnagscthm*

\begin{proof}
	(Step 1).
	In this step, $\langle\cdot,\cdot\rangle$ and $\Vert\cdot\Vert$ always denote the inner product and the norm on $T_{y_k}M$. Set $q=\mu s$. It is straightforward to check that $\grad f\left(y_{k}\right)=\mu\frac{1-\sqrt{q/\xi}}{\sqrt{q/\xi}}v_{k}-\mu\frac{1}{\sqrt{q/\xi}}\bar{\bar{v}}_{k+1}$
	and $\log_{y_{k}}\left(x_{k}\right)=-\sqrt{\xi q}v_{k}$.\footnote{Note that $y_{k}=\exp_{x_{k}}\left(\frac{\sqrt{\xi q}}{1+\sqrt{\xi q}}\bar{v}_{k}\right)$
		and $v_{k}=\Gamma_{x_{k}}^{y_{k}}\left(\bar{v}_{k}-\log_{x_{k}}\left(y_{k}\right)\right)=\Gamma_{x_{k}}^{y_{k}}\left(\left(1-\frac{\sqrt{\xi q}}{1+\sqrt{\xi q}}\right)\bar{v}_{k}\right)$.
		Let $\gamma_{1}$ be the geodesic such that $\gamma_{1}(0)=x_{k}$
		and $\gamma_{1}(1)=y_{k}$, then $\gamma_{1}'(0)=\log_{x_{k}}\left(y_{k}\right)$.
		Let $\gamma_{2}$ be the geodesic defined as $\gamma_{2}(t)=\gamma_{1}(1-t)$.
		Then $\log_{y_{k}}\left(x_{k}\right)=\gamma_{2}'(0)=-\gamma_{1}'(1)=-\Gamma_{x_{k}}^{y_{k}}\left(\gamma_{1}'(0)\right)=-\Gamma_{x_{k}}^{y_{k}}\left(\log_{x_{k}}\left(y_{k}\right)\right)$.
		Now, we have $\log_{y_{k}}\left(x_{k}\right)=-\Gamma_{x_{k}}^{y_{k}}\left(\log_{x_{k}}\left(y_{k}\right)\right)=-\frac{\sqrt{\xi q}}{1+\sqrt{\xi q}}\Gamma_{x_{k}}^{y_{k}}\left(\bar{v}_{k}\right)=-\frac{\frac{\sqrt{\xi q}}{1+\sqrt{\xi q}}}{1-\frac{\sqrt{\xi q}}{1+\sqrt{\xi q}}}{\color{red}v_k}={\color{red}-}\sqrt{\xi q}v_{k}$.} By geodesic $\mu$-strong convexity of $f$, we have
	\begin{align*}
		f\left(x^{*}\right) & \geq f\left(y_{k}\right)+\left\langle \grad f\left(y_{k}\right),\log_{y_{k}}\left(x^{*}\right)\right\rangle +\frac{\mu}{2}\left\Vert \log_{y_{k}}\left(x^{*}\right)\right\Vert ^{2}\\
		& =f\left(y_{k}\right)+\mu\frac{1-\sqrt{q/\xi}}{\sqrt{q/\xi}}\left\langle v_{k},\log_{y_{k}}\left(x^{*}\right)\right\rangle -\mu\frac{1}{\sqrt{q/\xi}}\left\langle \bar{\bar{v}}_{k+1},\log_{y_{k}}\left(x^{*}\right)\right\rangle +\frac{\mu}{2}\left\Vert \log_{y_{k}}\left(x^{*}\right)\right\Vert ^{2}.
	\end{align*}
	It follows from the geodesic convexity of $f$ that
	\begin{align*}
		f\left(x_{k}\right) & \geq f\left(y_{k}\right)+\left\langle \grad f\left(y_{k}\right),\log_{y_{k}}\left(x_{k}\right)\right\rangle \\
		& =f\left(y_{k}\right)-\xi\mu\left(1-\sqrt{\frac{q}{\xi}}\right)\left\Vert v_{k}\right\Vert ^{2}+\xi\mu\left\langle v_{k},\bar{\bar{v}}_{k+1}\right\rangle .
	\end{align*}
	By the geodesic $\frac{1}{s}$-smoothness of $f$, we have
	\begin{align*}
		f\left(x_{k+1}\right) & \leq f\left(y_{k}\right)+\left\langle \grad f\left(y_{k}\right),\log_{y_{k}}\left(x_{k+1}\right)\right\rangle +\frac{1}{2s}\left\Vert \log_{y_{k}}\left(x_{k+1}\right)\right\Vert ^{2}\\
		& =f\left(y_{k}\right)-\frac{s}{2}\left\Vert \grad f\left(y_k\right)\right\Vert^2\\
		& =f\left(y_{k}\right)-\frac{s}{2}\left\Vert \mu\frac{1-\sqrt{q/\xi}}{\sqrt{q/\xi}}v_{k}-\mu\frac{1}{\sqrt{q/\xi}}\bar{\bar{v}}_{k+1}\right\Vert ^{2}\\
		& =f\left(y_{k}\right)-\frac{\xi\mu}{2}\left(1-\sqrt{\frac{q}{\xi}}\right)^{2}\left\Vert v_{k}\right\Vert ^{2}+\xi\mu\left(1-\sqrt{\frac{q}{\xi}}\right)\left\langle v_{k},\bar{\bar{v}}_{k+1}\right\rangle -\frac{\xi\mu}{2}\left\Vert \bar{\bar{v}}_{k+1}\right\Vert ^{2}.
	\end{align*}
	Taking a weighted sum of these inequalities yields
	\begin{align*}
		0 & \geq\sqrt{\frac{q}{\xi}}\left[f\left(y_{k}\right)-f\left(x^{*}\right)+\mu\frac{1-\sqrt{q/\xi}}{\sqrt{q/\xi}}\left\langle v_{k},\log_{y_{k}}\left(x^{*}\right)\right\rangle -\mu\frac{1}{\sqrt{q/\xi}}\left\langle \bar{\bar{v}}_{k+1},\log_{y_{k}}\left(x^{*}\right)\right\rangle +\frac{\mu}{2}\left\Vert \log_{y_{k}}\left(x^{*}\right)\right\Vert ^{2}\right]\\
		& \quad+\left(1-\sqrt{\frac{q}{\xi}}\right)\left[f\left(y_{k}\right)-f\left(x_{k}\right)-\xi\mu\left(1-\sqrt{\frac{q}{\xi}}\right)\left\Vert v_{k}\right\Vert ^{2}+\xi\mu\left\langle v_{k},\bar{\bar{v}}_{k+1}\right\rangle \right]\\
		& \quad+\left[f\left(x_{k+1}\right)-f\left(y_{k}\right)+\frac{\xi\mu}{2}\left(1-\sqrt{\frac{q}{\xi}}\right)^{2}\left\Vert v_{k}\right\Vert ^{2}-\xi\mu\left(1-\sqrt{\frac{q}{\xi}}\right)\left\langle v_{k},\bar{\bar{v}}_{k+1}\right\rangle +\frac{\xi\mu}{2}\left\Vert \bar{\bar{v}}_{k+1}\right\Vert ^{2}\right]\\
		& =\left(f\left(x_{k+1}\right)-f\left(x^{*}\right)\right)-\left(1-\sqrt{\frac{q}{\xi}}\right)\left(f\left(x_{k}\right)-f\left(x^{*}\right)\right)\\
		& \quad+\mu\left(1-\sqrt{\frac{q}{\xi}}\right)\left\langle v_{k},\log_{y_{k}}\left(x^{*}\right)\right\rangle -\mu\left\langle \bar{\bar{v}}_{k+1},\log_{y_{k}}\left(x^{*}\right)\right\rangle +\frac{\mu}{2}\sqrt{\frac{q}{\xi}}\left\Vert \log_{y_{k}}\left(x^{*}\right)\right\Vert ^{2}\\
		& \quad-\frac{\xi\mu}{2}\left(1-\sqrt{\frac{q}{\xi}}\right)^{2}\left\Vert v_{k}\right\Vert ^{2}+\frac{\xi\mu}{2}\left\Vert \bar{\bar{v}}_{k+1}\right\Vert ^{2}.
	\end{align*}
	We further notice that
	\begin{align*}
		& \left\Vert \bar{\bar{v}}_{k+1}-\log_{y_{k}}\left(x^{*}\right)\right\Vert ^{2}-\left(1-\sqrt{\frac{q}{\xi}}\right)\left\Vert v_{k}-\log_{y_{k}}\left(x^{*}\right)\right\Vert ^{2}\\
		& =\left\Vert \bar{\bar{v}}_{k+1}\right\Vert ^{2}-2\left\langle \bar{\bar{v}}_{k+1},\log_{y_{k}}\left(x^{*}\right)\right\rangle -\left(1-\sqrt{\frac{q}{\xi}}\right)\left\Vert v_{k}\right\Vert ^{2}+2\left(1-\sqrt{\frac{q}{\xi}}\right)\left\langle v_{k},\log_{y_{k}}\left(x^{*}\right)\right\rangle +\sqrt{\frac{q}{\xi}}\left\Vert \log_{y_{k}}\left(x^{*}\right)\right\Vert ^{2}.
	\end{align*}
	Therefore, we obtain
	\begin{align*}
		0 & \geq\left(f\left(x_{k+1}\right)-f\left(x^{*}\right)+\frac{\mu}{2}\left\Vert \bar{\bar{v}}_{k+1}-\log_{y_{k}}\left(x^{*}\right)\right\Vert ^{2}\right)-\left(1-\sqrt{\frac{q}{\xi}}\right)\left(f\left(x_{k}\right)-f\left(x^{*}\right)+\frac{\mu}{2}\left\Vert v_{k}-\log_{y_{k}}\left(x^{*}\right)\right\Vert ^{2}\right)\\
		& \quad+(\xi-1)\frac{\mu}{2}\left\Vert \bar{\bar{v}}_{k+1}\right\Vert ^{2}-\frac{\xi\mu}{2}\left(1-\sqrt{\frac{q}{\xi}}\right)^{2}\left\Vert v_{k}\right\Vert ^{2}+\frac{\mu}{2}\left(1-\sqrt{\frac{q}{\xi}}\right)\left\Vert v_{k}\right\Vert ^{2}\\
		& =\left(f\left(x_{k+1}\right)-f\left(x^{*}\right)+\frac{\mu}{2}\left\Vert \bar{\bar{v}}_{k+1}-\log_{y_{k}}\left(x^{*}\right)\right\Vert ^{2}\right)-\left(1-\sqrt{\frac{q}{\xi}}\right)\left(f\left(x_{k}\right)-f\left(x^{*}\right)+\frac{\mu}{2}\left\Vert v_{k}-\log_{y_{k}}\left(x^{*}\right)\right\Vert ^{2}\right)\\
		& \quad+(\xi-1)\frac{\mu}{2}\left\Vert \bar{\bar{v}}_{k+1}\right\Vert ^{2}-(\xi-1)\frac{\mu}{2}\left(1-\sqrt{\frac{q}{\xi}}\right)\left\Vert v_{k}\right\Vert ^{2}+\frac{\xi\mu}{2}\sqrt{\frac{q}{\xi}}\left(1-\sqrt{\frac{q}{\xi}}\right)\left\Vert v_{k}\right\Vert ^{2}\\
		& =\left(f\left(x_{k+1}\right)-f\left(x^{*}\right)+\frac{\mu}{2}\left\Vert \bar{\bar{v}}_{k+1}-\log_{y_{k}}\left(x^{*}\right)\right\Vert ^{2}+(\xi-1)\frac{\mu}{2}\left\Vert \bar{\bar{v}}_{k+1}\right\Vert ^{2}\right)\\
		& \quad-\left(1-\sqrt{\frac{q}{\xi}}\right)\left(f\left(x_{k}\right)-f\left(x^{*}\right)+\frac{\mu}{2}\left\Vert v_{k}-\log_{y_{k}}\left(x^{*}\right)\right\Vert ^{2}+(\xi-1)\frac{\mu}{2}\left\Vert v_{k}\right\Vert ^{2}\right)\\
		& \quad+\frac{\xi\mu}{2}\sqrt{\frac{q}{\xi}}\left(1-\sqrt{\frac{q}{\xi}}\right)\left\Vert v_{k}\right\Vert ^{2}.
	\end{align*}
	
	(Step 2: Handle metric distortion).
	It follows from \cref{lem:distortion2} with $p_A=y_k$, $p_B=x_{k+1}$, $x=x^{*}$, $v_A=\bar{\bar{v}}_{k+1}$, $v_B=\bar{v}_{k+1}$, $a=\left(1-\sqrt{\frac{q}{\xi}}\right)v_k$, $b=\sqrt{\frac{q}{\xi}}\left(-\frac{1}{\mu}\right)\grad f\left(y_k\right)$, $r=\sqrt{\xi q}$ that
	\begin{align*}
		& \left\Vert \log_{x_{k+1}}\left(x^{*}\right)-\bar{v}_{k+1}\right\Vert_{x_{k+1}} ^{2}+(\xi-1)\left\Vert \bar{v}_{k+1}\right\Vert_{x_{k+1}} ^{2}\\
		& \leq\left\Vert \log_{y_{k}}\left(x^{*}\right)-\bar{\bar{v}}_{k+1}\right\Vert_{y_k} ^{2}+(\xi-1)\left\Vert \bar{\bar{v}}_{k+1}\right\Vert_{y_k} ^{2}+\frac{\xi-\delta}{2}\left(\frac{1}{1-\sqrt{\xi q}}-1\right)\left\Vert \left(1-\sqrt{\frac{q}{\xi}}\right)v_{k}\right\Vert_{y_k} ^{2}.
	\end{align*}
	Applying \cref{lem:distortion1} with $p_A=x_{k}$, $p_B=y_{k}$, $x=x^{*}$, $v_A=\bar{v}_{k}$, $v_B=v_{k}$, $r=\frac{\sqrt{\xi q}}{1+\sqrt{\xi q}}$ gives
	\begin{align*}
		& \left\Vert \log_{x_{k}}\left(x^{*}\right)-\bar{v}_{k}\right\Vert _{x_{k}}^{2}+(\xi-1)\left\Vert \bar{v}_{k}\right\Vert _{x_{k}}^{2}\\
		& =\left(\left\Vert \log_{x_{k}}\left(x^{*}\right)-\bar{v}_{k}\right\Vert _{x_{k}}^{2}+(\zeta-1)\left\Vert \bar{v}_{k}\right\Vert _{x_{k}}^{2}\right)+(\xi-\zeta)\left\Vert \bar{v}_{k}\right\Vert _{x_{k}}^{2}\\
		& \geq\left(\left\Vert \log_{y_{k}}\left(x^{*}\right)-v_{k}\right\Vert _{y_{k}}^{2}+(\zeta-1)\left\Vert v_{k}\right\Vert _{y_{k}}^{2}\right)+(\xi-\zeta)\left\Vert \bar{v}_{k}\right\Vert _{x_{k}}^{2}\\
		& =\left\Vert \log_{y_{k}}\left(x^{*}\right)-v_{k}\right\Vert _{y_{k}}^{2}+(\zeta-1)\left\Vert v_{k}\right\Vert _{y_{k}}^{2}+(\xi-\zeta)\frac{1}{\left(1-\frac{\sqrt{\xi q}}{1+\sqrt{\xi q}}\right)^{2}}\left\Vert v_{k}\right\Vert _{y_{k}}^{2}\\
		& =\left\Vert \log_{y_{k}}\left(x^{*}\right)-v_{k}\right\Vert _{y_{k}}^{2}+(\xi-1)\left\Vert v_{k}\right\Vert _{y_{k}}^{2}+(\xi-\zeta)\left(\frac{1}{\left(1-\frac{\sqrt{\xi q}}{1+\sqrt{\xi q}}\right)^{2}}-1\right)\left\Vert v_{k}\right\Vert _{y_{k}}^{2}
	\end{align*}
	
	Combining these inequalities with the result in Step~1 gives
	\begin{align*}
		0 & \geq\left(f\left(x_{k+1}\right)-f\left(x^{*}\right)+\frac{\mu}{2}\left\Vert \bar{\bar{v}}_{k+1}-\log_{y_{k}}\left(x^{*}\right)\right\Vert _{y_{k}}^{2}+(\xi-1)\frac{\mu}{2}\left\Vert \bar{\bar{v}}_{k+1}\right\Vert _{y_{k}}^{2}\right)\\
		& \quad-\left(1-\sqrt{\frac{q}{\xi}}\right)\left(f\left(x_{k}\right)-f\left(x^{*}\right)+\frac{\mu}{2}\left\Vert v_{k}-\log_{y_{k}}\left(x^{*}\right)\right\Vert _{y_{k}}^{2}+(\xi-1)\frac{\mu}{2}\left\Vert v_{k}\right\Vert _{y_{k}}^{2}\right)\\
		& \quad+\frac{\mu}{2}\sqrt{\xi q}\left(1-\sqrt{\frac{q}{\xi}}\right)\left\Vert v_{k}\right\Vert _{y_{k}}^{2}\\
		& \quad+\frac{\mu}{2}\left[\left\Vert \log_{x_{k+1}}\left(x^{*}\right)-\bar{v}_{k+1}\right\Vert _{x_{k+1}}^{2}+(\xi-1)\left\Vert \bar{v}_{k+1}\right\Vert _{x_{k+1}}^{2}\right.\\
		& \quad\quad-\left.\left\Vert \log_{y_{k}}\left(x^{*}\right)-\bar{\bar{v}}_{k+1}\right\Vert _{y_{k}}^{2}-(\xi-1)\left\Vert \bar{\bar{v}}_{k+1}\right\Vert _{y_{k}}^{2}-\frac{\xi-\delta}{2}\left(\frac{1}{1-\sqrt{\xi q}}-1\right)\left\Vert \left(1-\sqrt{\frac{q}{\xi}}\right)v_{k}\right\Vert _{y_{k}}^{2}\right]\\
		& \quad+\frac{\mu}{2}{\color{red}\left(1-\sqrt{\frac{q}{\xi}}\right)}\left[\left\Vert \log_{y_{k}}\left(x^{*}\right)-v_{k}\right\Vert _{y_{k}}^{2}+(\xi-1)\left\Vert v_{k}\right\Vert _{y_{k}}^{2}+(\xi-\zeta)\left(\frac{1}{\left(1-\frac{\sqrt{\xi q}}{1+\sqrt{\xi q}}\right)^{2}}-1\right)\left\Vert v_{k}\right\Vert _{y_{k}}^{2}\right.\\
		& \quad\quad-\left.\left\Vert \log_{x_{k}}\left(x^{*}\right)-\bar{v}_{k}\right\Vert _{x_{k}}^{2}-(\xi-1)\left\Vert \bar{v}_{k}\right\Vert _{x_{k}}^{2}\right]\\
		& =\left(1-\sqrt{\frac{q}{\xi}}\right)^{k+1}\left(\phi_{k+1}-\phi_{k}\right)\\
		& \quad+\frac{\mu}{2}\left(\sqrt{\xi q}\left(1-\sqrt{\frac{q}{\xi}}\right)-\frac{\xi-\delta}{2}\left(\frac{1}{1-\sqrt{\xi q}}-1\right)\left(1-\sqrt{\frac{q}{\xi}}\right)^{2}+(\xi-\zeta){\color{red}\left(1-\sqrt{\frac{q}{\xi}}\right)}\left(\frac{1}{\left(1-\frac{\sqrt{\xi q}}{1+\sqrt{\xi q}}\right)^{2}}-1\right)\right)\left\Vert v_{k}\right\Vert ^{2}\\
		& \geq\left(1-\sqrt{\frac{q}{\xi}}\right)^{k+1}\left(\phi_{k+1}-\phi_{k}\right).
	\end{align*}
\end{proof}

\rnagsccor*

\begin{proof}
	(Step 1: Checking the condition for \cref{thm:rnag-sc}).
	
	It is straightforward to check that
	\[
	\frac{\xi-\delta}{2}\left(\frac{1}{1-t}-1\right)\leq(\xi-\zeta)\left(\frac{1}{1-\frac{t}{1+t}}-1\right)
	\]
	for all $t\in(0,1/3]$. Because $\sqrt{\xi q}=\sqrt{\xi\mu\frac{1}{9\xi L}}=\frac{1}{3}\sqrt{\mu/L}\in(0,1/3]$, we have
	\begin{align*}
		(\xi-\zeta){\color{red}\left(1-\sqrt{\frac{q}{\xi}}\right)}\left(\frac{1}{\left(1-\frac{\sqrt{\xi q}}{1+\sqrt{\xi q}}\right)^{2}}-1\right) & \geq(\xi-\zeta){\color{red}\left(1-\sqrt{\frac{q}{\xi}}\right)}\left(\frac{1}{\left(1-\frac{\sqrt{\xi q}}{1+\sqrt{\xi q}}\right)}-1\right)\\
		& \geq\frac{\xi-\delta}{2}{\color{red}\left(1-\sqrt{\frac{q}{\xi}}\right)}\left(\frac{1}{1-\sqrt{\xi q}}-1\right).
	\end{align*}
	Because $\sqrt{\frac{q}{\xi}}\in (0,1)$, we have
	\begin{align*}
		\frac{\xi-\delta}{2}\left(\frac{1}{1-\sqrt{\xi q}}-1\right)\left(1-\sqrt{\frac{q}{\xi}}\right)^{2}-\sqrt{\xi q}\left(1-\sqrt{\frac{q}{\xi}}\right) & \leq\frac{\xi-\delta}{2}\left(\frac{1}{1-\sqrt{\xi q}}-1\right)\left(1-\sqrt{\frac{q}{\xi}}\right)^{2}\\
		& \leq\frac{\xi-\delta}{2}{\color{red}\left(1-\sqrt{\frac{q}{\xi}}\right)}\left(\frac{1}{1-\sqrt{\xi q}}-1\right).
	\end{align*}
	Combining these inequalities gives the desired condition.
	
	(Step 2: Computing iteration complexity).
	It follows from \cref{thm:rnag-c} that
	\[
	f\left(x_{k}\right)-f\left(x^{*}\right)\leq\left(1-\sqrt{\frac{q}{\xi}}\right)^{k}\phi_{k}\leq\left(1-\sqrt{\frac{q}{\xi}}\right)^{k}\phi_{0}=\left(1-\sqrt{\frac{q}{\xi}}\right)^{k}\left(f\left(x_{0}\right)-f\left(x^{*}\right)+\frac{\mu}{2}\left\Vert \log_{x_{0}}\left(x^{*}\right)\right\Vert ^{2}\right).
	\]
	By the geodesic $L$-smoothness of $f$, we have
	\begin{align*}
		f\left(x_{k}\right)-f\left(x^{*}\right) & \leq\left(1-\sqrt{\frac{q}{\xi}}\right)^{k}\left(\frac{L}{2}\left\Vert \log_{x_{0}}\left(x^{*}\right)\right\Vert ^{2}+\frac{\mu}{2}\left\Vert \log_{x_{0}}\left(x^{*}\right)\right\Vert ^{2}\right)\\
		& \leq\left(1-\sqrt{\frac{q}{\xi}}\right)^{k}L\left\Vert \log_{x_{0}}\left(x^{*}\right)\right\Vert ^{2}\\
		& =\left(1-\sqrt{\frac{\mu}{9\xi^{2}L}}\right)^{k}L\left\Vert \log_{x_{0}}\left(x^{*}\right)\right\Vert ^{2}\\
		& \leq e^{-\sqrt{\frac{\mu}{9\xi^{2}L}}k}L\left\Vert \log_{x_{0}}\left(x^{*}\right)\right\Vert ^{2}\\
		& {\color{red}\raisebox{1ex}{\underline{\smash{\raisebox{-1ex}{$\leq e^{-\sqrt{\frac{\mu}{9\xi^{2}L}}k}L\left\Vert \log_{x_{0}}\left(x^{*}\right)\right\Vert ^{2}.$}}}}}
	\end{align*}
	Thus, we have $f\left(x_{k}\right)-f\left(x^{*}\right)\leq\epsilon$ whenever
	\[
	k\geq\sqrt{\frac{9\xi^{2}L}{\mu}}\log\left(\frac{L}{\epsilon}\left\Vert \log_{x_{0}}\left(x^{*}\right)\right\Vert ^{2}\right),
	\]
	which implies the $O\left(\xi\sqrt{\frac{L}{\mu}}\log\frac{L}{\epsilon}\right)$ iteration complexity of RNAG-SC.
\end{proof}

\section{Continuous-Time Interpretation}
\label{app:continuous}

\subsection{The g-convex case}

Because we approximate the curve $y(t)$ by the iterates $y_k$, we first rewrite RNAG-C in the form using only the iterates $y_k$ as follows:
\begin{align*}
	y_{k+1}-y_{k} & =x_{k+1}-y_{k}+\frac{\xi}{\lambda_{k+1}+\xi-1}\bar{v}_{k+1}\\
	& =-s\grad f\left(y_{k}\right)+\frac{\xi}{\lambda_{k+1}+\xi-1}\left(\bar{\bar{v}}_{k+1}+s\grad f\left(y_{k}\right)\right)\\
	& =-s\grad f\left(y_{k}\right)+\frac{\xi}{\lambda_{k+1}+\xi-1}\left(v_{k}-\frac{s\lambda_{k}}{\xi}\grad f\left(y_{k}\right)+s\grad f\left(y_{k}\right)\right)\\
	& =\left(-1+\frac{-\lambda_{k}+\xi}{\lambda_{k-1}+\xi-1}\right)s\grad f\left(y_{k}\right)+\frac{\xi}{\lambda_{k+1}+(\xi-1)}\frac{\lambda_{k}-1}{\xi}\left(y_{k}-x_{k}\right)\\
	& =\frac{1-\lambda_{k}-\lambda_{k+1}}{\lambda_{k+1}+(\xi-1)}s\grad f\left(y_{k}\right)+\frac{\lambda_{k}-1}{\lambda_{k+1}+(\xi-1)}\left(y_{k}-x_{k}\right)\\
	& =\frac{1-\lambda_{k}-\lambda_{k+1}}{\lambda_{k+1}+(\xi-1)}s\grad f\left(y_{k}\right)+\frac{\lambda_{k}-1}{\lambda_{k+1}+(\xi-1)}\left(y_{k}-y_{k-1}+s\grad f\left(y_{k-1}\right)\right)\\
	& =\frac{\lambda_{k}-1}{\lambda_{k+1}+(\xi-1)}\left(y_{k}-y_{k-1}\right)-\frac{\lambda_{k+1}}{\lambda_{k+1}+(\xi-1)}s\grad f\left(y_{k}\right)\\
	& \quad+\frac{\lambda_{k}-1}{\lambda_{k+1}+(\xi-1)}s\left(\grad f\left(y_{k-1}\right)-\grad f\left(y_{k}\right)\right)
\end{align*}
We introduce a smooth curve $y(t)$ as mentioned in \cref{sec:continuous}. Now, dividing both sides of the above equality by $\sqrt{s}$ and substituting 
\begin{align*}
	\frac{y_{k+1}-y_{k}}{\sqrt{s}} & =\dot{y}+\frac{\sqrt{s}}{2}\ddot{y}+o\left(\sqrt{s}\right)\\
	\frac{y_{k}-y_{k-1}}{\sqrt{s}} & =\dot{y}-\frac{\sqrt{s}}{2}\ddot{y}+o\left(\sqrt{s}\right)\\
	\sqrt{s}\grad f\left(y_{k-1}\right) & =\sqrt{s}\grad f\left(y_{k}\right)+o\left(\sqrt{s}\right),
\end{align*}
we obtain
\[
\dot{y}+\frac{\sqrt{s}}{2}\ddot{y}+o\left(\sqrt{s}\right)=\frac{\lambda_{k}-1}{\lambda_{k+1}+(\xi-1)}\left(\dot{y}-\frac{\sqrt{s}}{2}\ddot{y}+o\left(\sqrt{s}\right)\right)-\frac{\lambda_{k+1}}{\lambda_{k+1}+(\xi-1)}\sqrt{s}\grad f(y).
\]
Dividing both sides by $\sqrt{s}$ and rearranging terms, we have
\[
\frac{1}{2}\left(1+\frac{\lambda_{k}-1}{\lambda_{k+1}+(\xi-1)}\right)\ddot{y}+\frac{1}{\sqrt{s}}\left(1-\frac{\lambda_{k}-1}{\lambda_{k+1}+(\xi-1)}\right)\dot{y}+\frac{\lambda_{k+1}}{\lambda_{k+1}+(\xi-1)}\grad f(y)+\frac{o\left(\sqrt{s}\right)}{\sqrt{s}}=0.
\]
Substituting $k=\frac{t}{\sqrt{s}}$,  we can check that $\frac{\lambda_{k}-1}{\lambda_{k+1}+(\xi-1)}\rightarrow1$,
$\frac{\lambda_{k+1}}{\lambda_{k+1}+(\xi-1)}\rightarrow1$, and $\frac{1}{\sqrt{s}}\left(1-\frac{\lambda_{k}-1}{\lambda_{k+1}+(\xi-1)}\right)=\frac{1}{\sqrt{s}}\frac{\lambda_{k+1}-\lambda_{k}+\xi}{\lambda_{k+1}+(\xi-1)}=\frac{1}{\sqrt{s}}\frac{1+2\xi}{k+T+4\xi-{\color{red}1}}=\frac{1+2\xi}{t+(T+4\xi-{\color{red}1})\sqrt{s}}\rightarrow\frac{1+2\xi}{t}$ as $s\rightarrow0$. Therefore, we obtain
\[
\ddot{y}+\frac{1+2\xi}{t}\dot{y}+\grad f(y)=0.
\]

\subsection{The g-strongly convex case}

As we approximate the curve $y(t)$ by the iterates $y_k$, we first rewrite RNAG-C in the form using only the iterates $y_k$ as follows:
\begin{align*}
	y_{k+1}-y_{k} & =x_{k+1}-y_{k}+\frac{\sqrt{\xi\mu s}}{1+\sqrt{\xi\mu s}}\bar{v}_{k+1}\\
	& =-s\grad f\left(y_{k}\right)+\frac{\sqrt{\xi\mu s}}{1+\sqrt{\xi\mu s}}\left(\bar{\bar{v}}_{k+1}+s\grad f\left(y_{k}\right)\right)\\
	& =-\frac{s}{1+\sqrt{\xi\mu s}}\grad f\left(y_{k}\right)+\frac{\sqrt{\xi\mu s}}{1+\sqrt{\xi\mu s}}\left(\left(1-\sqrt{\frac{\mu s}{\xi}}\right)v_{k}+\sqrt{\frac{\mu s}{\xi}}\left(-\frac{\grad f\left(y_{k}\right)}{\mu}\right)\right)\\
	& =-\frac{2s}{1+\sqrt{\xi\mu s}}\grad f\left(y_{k}\right)+\frac{\sqrt{\xi\mu s}}{1+\sqrt{\xi\mu s}}\left(1-\sqrt{\frac{\mu s}{\xi}}\right)\frac{1}{\sqrt{\xi\mu s}}\left(y_{k}-x_{k}\right)\\
	& =-\frac{2s}{1+\sqrt{\xi\mu s}}\grad f\left(y_{k}\right)+\frac{1-\sqrt{\mu s/\xi}}{1+\sqrt{\xi\mu s}}\left(y_{k}-y_{k-1}+s\grad f\left(y_{k-1}\right)\right)\\
	& =\frac{1-\sqrt{\mu s/\xi}}{1+\sqrt{\xi\mu s}}\left(y_{k}-y_{k-1}\right)-\frac{1+\sqrt{\mu s/\xi}}{1+\sqrt{\xi\mu s}}s\grad f\left(y_{k}\right)+\frac{1-\sqrt{\mu s/\xi}}{1+\sqrt{\xi\mu s}}s\left(\grad f\left(y_{k-1}\right)-\grad f\left(y_{k}\right)\right)
\end{align*}
Dividing both sides by $\sqrt{s}$ and substituting
\begin{align*}
	\frac{y_{k+1}-y_{k}}{\sqrt{s}} & =\dot{y}+\frac{\sqrt{s}}{2}\ddot{y}+o\left(\sqrt{s}\right)\\
	\frac{y_{k}-y_{k-1}}{\sqrt{s}} & =\dot{y}-\frac{\sqrt{s}}{2}\ddot{y}+o\left(\sqrt{s}\right)\\
	\sqrt{s}\grad f\left(y_{k-1}\right) & =\sqrt{s}\grad f\left(y_{k}\right)+o\left(\sqrt{s}\right)
\end{align*}
yield
\[
\dot{y}+\frac{\sqrt{s}}{2}\ddot{y}+o\left(\sqrt{s}\right)=\frac{1-\sqrt{\mu s/\xi}}{1+\sqrt{\xi\mu s}}\left(\dot{y}-\frac{\sqrt{s}}{2}\ddot{y}+o\left(\sqrt{s}\right)\right)-\frac{1+\sqrt{\mu s/\xi}}{1+\sqrt{\xi\mu s}}\sqrt{s}\grad f\left(y_{\color{red}\raisebox{1ex}{\underline{\smash{\raisebox{-1ex}{\scriptsize $k$}}}}}\right).
\]
Dividing both sides by $\sqrt{s}$ and rearranging terms, we obtain
\[
\frac{1}{2}\left(1+\frac{1-\sqrt{\mu s/\xi}}{1+\sqrt{\xi\mu s}}\right)\ddot{y}+\frac{\left(\sqrt{1/\xi}+\sqrt{\xi}\right)\sqrt{\mu}}{1+\sqrt{\xi\mu s}}\dot{y}+\frac{1+\sqrt{\mu s/\xi}}{1+\sqrt{\xi\mu s}}\grad f\left(y_{k}\right)+\frac{o\left(\sqrt{s}\right)}{\sqrt{s}}=0.
\]
Taking the limit $s\rightarrow0$ gives
\[
\ddot{y}+\left(\frac{1}{\sqrt{\xi}}+\sqrt{\xi}\right)\sqrt{\mu}\dot{y}+\grad f(y)=0
\]
as desired.

\subsection{Experiments}

In this section, we empirically show that the iterates of our methods converge to the solution of the corresponding ODEs, as taking the limit $s\rightarrow0$. We use the Rayleigh quotient maximization problem in \cref{sec:experiments} with $d=10$ and $\xi=2$. For RNAG-SC, we set $\mu=0.1$ (note that the limiting argument above does not use geodesic $\mu$-strong convexity of $f$). To compute the solution of ODEs \eqref{eq:rnag-ode-c} and \eqref{eq:rnag-ode-sc}, we implement SIRNAG (Option I) \citep{alimisis2020continuous} with very small integration step size. The results are shown in \cref{fig:ode-c} and \cref{fig:ode-sc}.

\begin{figure*}[ht]
	\centering
	\subfigure[Error vs. \# of iterations]{\includegraphics[width=0.45\textwidth]{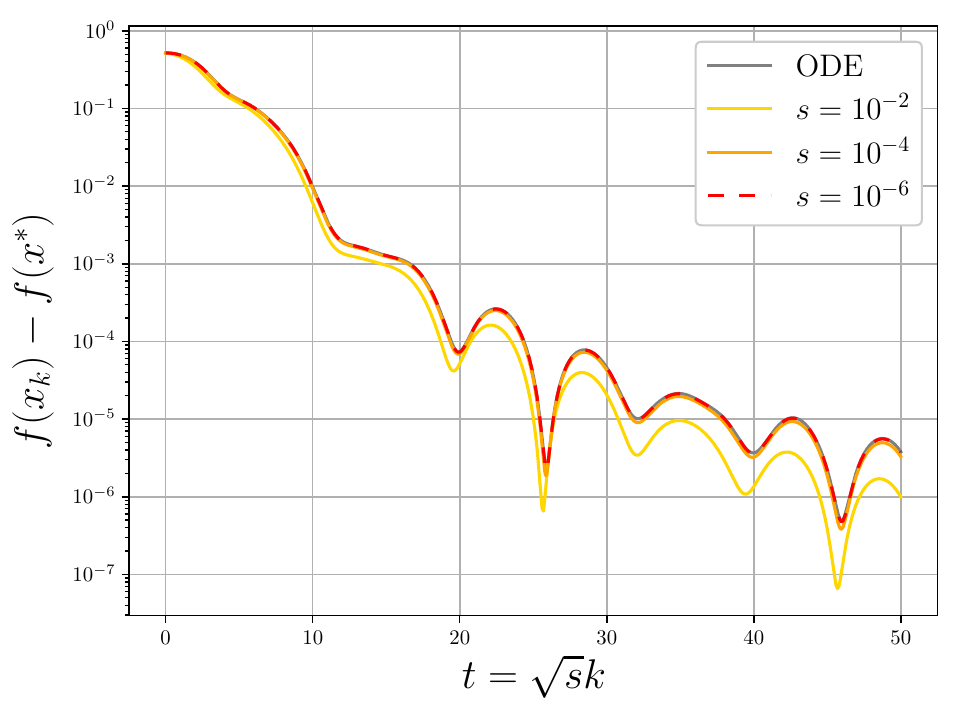}\label{fig:ode-c1}}
	\subfigure[Distance vs. \# of iterations]{\includegraphics[width=0.45\textwidth]{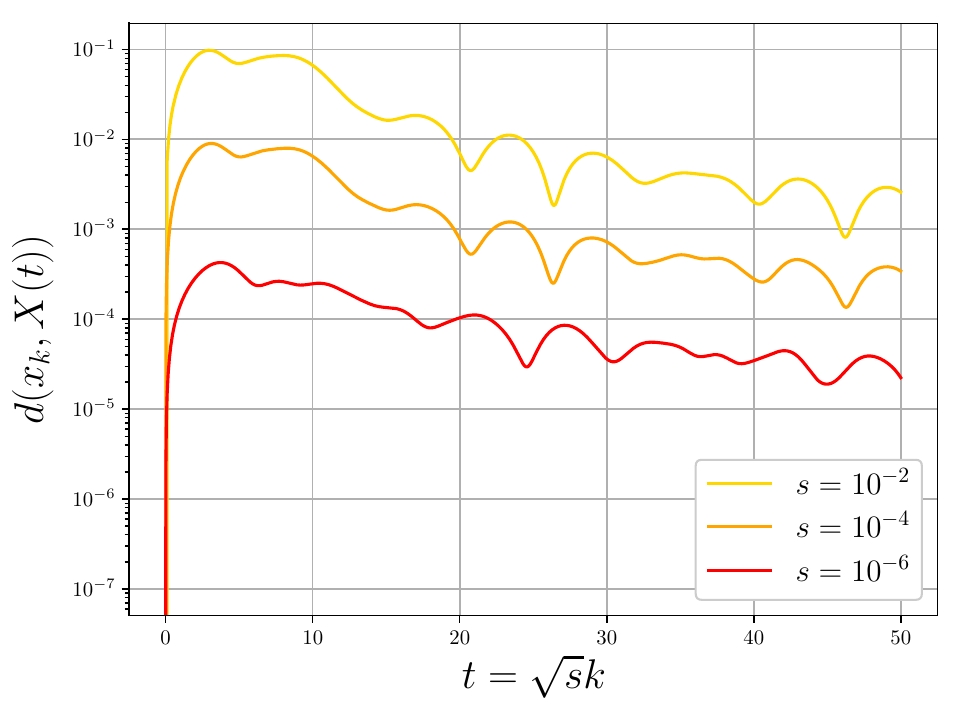}\label{fig:ode-c2}}
	\caption{Convergence of RNAG-C to the solution of ODE \eqref{eq:rnag-ode-c}.}\label{fig:ode-c}
\end{figure*}

\begin{figure*}[ht]
	\centering
	\subfigure[Error vs. \# of iterations]{\includegraphics[width=0.45\textwidth]{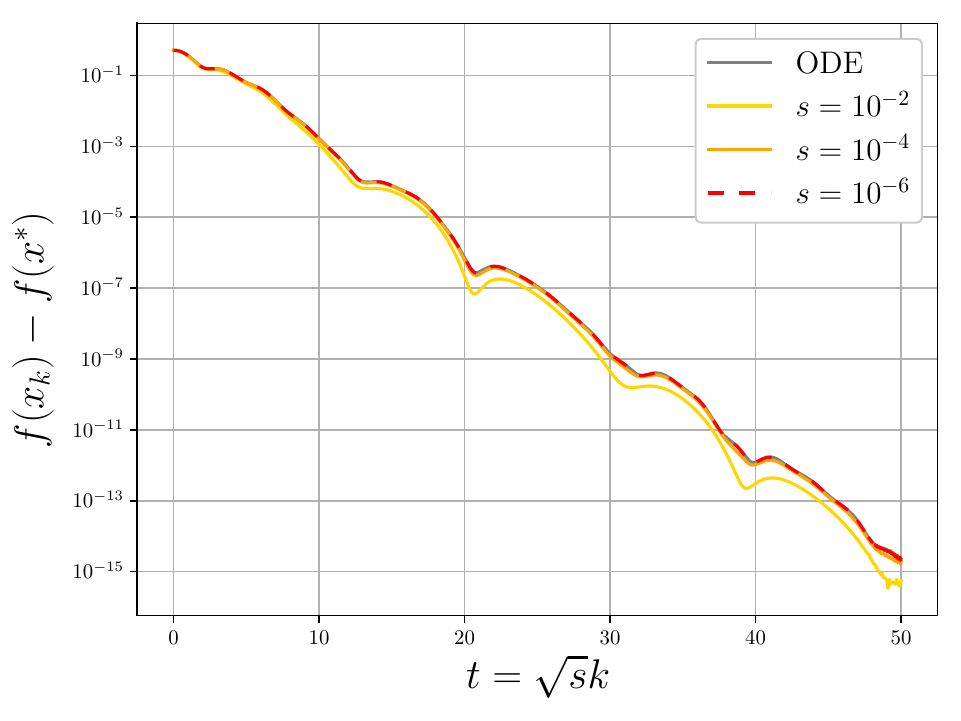}\label{fig:ode-sc1}}
	\subfigure[Distance vs. \# of iterations]{\includegraphics[width=0.45\textwidth]{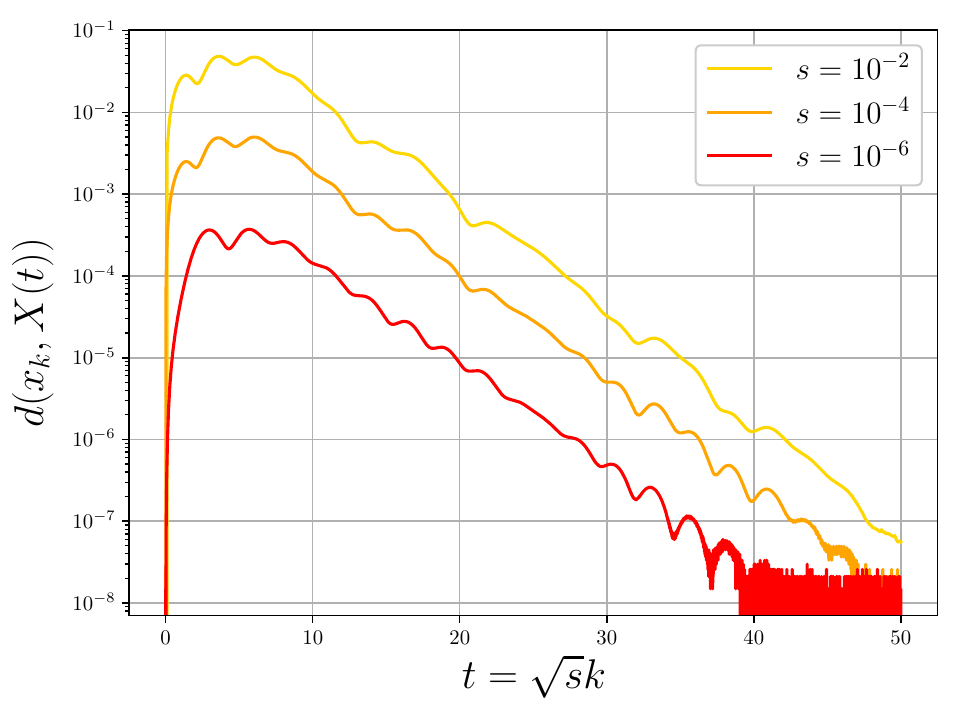}\label{fig:ode-sc2}}
	\caption{Convergence of RNAG-SC to the solution of ODE \eqref{eq:rnag-ode-sc}.}\label{fig:ode-sc}
\end{figure*}

\section{Proofs for \cref{sec:experiments}}
\label{app:exp}

\experimenta*

\begin{proof}
	For $x\in\mathbb{S}^{d-1}\subseteq\mathbb{R}^d$ and a unit tangent vector $v\in T_{x}M$,
	we have
	\[
	\exp_{x}(tv)=\frac{x+\tan(t)\,v}{\left\Vert x+\tan(t)\,v\right\Vert }=\frac{x+\tan(t)\,v}{\sec(t)}
	\]
	for $t\in I$, where $I$ is a small interval containing $0$.
	We consider the function $h:I\rightarrow M$ defined as
	\begin{align*}
		h(t) & =f\left(\exp_{x}(tv)\right)\\
		& =-\frac{1}{2}\cos^{2}(t)\left(x+\tan(t)\,v\right)^{\top}A\left(x+\tan(t)\,v\right)\\
		& =-\frac{1}{2}h_{1}(t)h_{2}(t),
	\end{align*}
	where $h_{1}(t)=\cos^{2}(t)$ and $h_{2}(t)=\left(x+\tan(t)\,v\right)^{\top}A\left(x+\tan(t)\,v\right)$.
	Note that $h_{1}(0)=1$, $h_{1}'(0)=0$, $h_{1}''(0)=-2$, $h_{2}(0)=x^{\top}Ax$, $h_{2}'(0)=2v^{\top}Ax$, and $h_{2}''(0)=2v^{\top}Av$.
	Now, by the product rule, we have
	\[
	h''(0)=-\frac{1}{2}h_{1}''(0)h_{2}(0)-h_{1}'(0)h_{2}'(0)-\frac{1}{2}h_{1}(0)h_{2}''(0)=x^{\top}Ax-v^{\top}Av.
	\]
	Because Rayleigh quotient is always in $\left[\lambda_{\min},\lambda_{\max}\right]$,
	we have $\vert h''(0) \vert \leq\left(\lambda_{\max}-\lambda_{\min}\right)$. This
	shows that $f$ is geodesically $\left(\lambda_{\max}-\lambda_{\min}\right)$-smooth.
\end{proof}

\experimentb*

\begin{proof}
	It is enough to show that the function $x\mapsto d\left(x,p_i\right)^2$ is {\color{red}geodesically $2$-strongly} convex. When $K_{\max}\leq 0$, we have $\delta=1$. Let $\gamma:I\rightarrow M$ be a geodesic whose image is in $N$. It follows from \cref{prop:alimisis2} that
	\begin{align*}
		\frac{d^{2}}{dt^{2}}\frac{1}{2}d\left(\gamma(t),p_{i}\right)^{2} & =\frac{d}{dt}\left\langle \log_{\gamma(t)}\left(p_{i}\right),-\gamma'(t)\right\rangle \\
		& =\left\langle D_{t}\log_{\gamma(t)}\left(p_{i}\right),-\gamma'(t)\right\rangle +\left\langle \log_{\gamma(t)}\left(p_{i}\right),-\gamma''(t)\right\rangle .
	\end{align*}
	Note that $\gamma''(t)=0$ because $\gamma$ is a geodesic. Now, \cref{prop:alimisis} gives $\frac{d^{2}}{dt^{2}}\frac{1}{2}d\left(\gamma(t),p_{i}\right)^{2}\geq1$.
\end{proof}

\end{document}